\documentclass[english,10pt,a4paper]{amsart}
\usepackage[T1]{fontenc}
\usepackage[left=2cm, right=2cm, top=2cm, bottom=2cm]{geometry}
\usepackage{amssymb}
\usepackage{amsthm,amsmath}
\usepackage{xcolor}
\usepackage{babel}

\usepackage{physics}
\usepackage{amsmath}
\usepackage{tikz}
\usepackage{mathdots}
\usepackage{yhmath}
\usepackage{cancel}
\usepackage{color}
\usepackage{siunitx}
\usepackage{array}
\usepackage{multirow}
\usepackage{amssymb}
\usepackage{gensymb}
\usepackage{tabularx}
\usepackage{extarrows}
\usepackage{booktabs}
\usetikzlibrary{fadings}
\usetikzlibrary{patterns}
\usetikzlibrary{shadows.blur}
\usetikzlibrary{shapes}
\usetikzlibrary{arrows.meta,decorations.markings,calc}
\usepackage{pgfplots}
\pgfplotsset{compat=1.17}
\usetikzlibrary{arrows.meta}

\tikzset{
	pattern size/.store in=\mcSize, 
	pattern size = 5pt,
	pattern thickness/.store in=\mcThickness, 
	pattern thickness = 0.3pt,
	pattern radius/.store in=\mcRadius, 
	pattern radius = 1pt}
\makeatletter
\pgfutil@ifundefined{pgf@pattern@name@_kqktknjhq}{
	\pgfdeclarepatternformonly[\mcThickness,\mcSize]{_kqktknjhq}
	{\pgfqpoint{0pt}{0pt}}
	{\pgfpoint{\mcSize+\mcThickness}{\mcSize+\mcThickness}}
	{\pgfpoint{\mcSize}{\mcSize}}
	{
		\pgfsetcolor{\tikz@pattern@color}
		\pgfsetlinewidth{\mcThickness}
		\pgfpathmoveto{\pgfqpoint{0pt}{0pt}}
		\pgfpathlineto{\pgfpoint{\mcSize+\mcThickness}{\mcSize+\mcThickness}}
		\pgfusepath{stroke}
}}
\makeatother

\tikzset{
	pattern size/.store in=\mcSize, 
	pattern size = 5pt,
	pattern thickness/.store in=\mcThickness, 
	pattern thickness = 0.3pt,
	pattern radius/.store in=\mcRadius, 
	pattern radius = 1pt}
\makeatletter
\pgfutil@ifundefined{pgf@pattern@name@_j16cwztqz}{
	\pgfdeclarepatternformonly[\mcThickness,\mcSize]{_j16cwztqz}
	{\pgfqpoint{0pt}{0pt}}
	{\pgfpoint{\mcSize+\mcThickness}{\mcSize+\mcThickness}}
	{\pgfpoint{\mcSize}{\mcSize}}
	{
		\pgfsetcolor{\tikz@pattern@color}
		\pgfsetlinewidth{\mcThickness}
		\pgfpathmoveto{\pgfqpoint{0pt}{0pt}}
		\pgfpathlineto{\pgfpoint{\mcSize+\mcThickness}{\mcSize+\mcThickness}}
		\pgfusepath{stroke}
}}
\makeatother
\tikzset{every picture/.style={line width=0.75pt}} 

\newtheorem{lemma}{Lemma}
\newtheorem{prop}{Proposition}
\newtheorem{thm}{Theorem}
\newtheorem{conjecture}{Conjecture}
\newtheorem{corollary}{Corollary}

\theoremstyle{definition}
\newtheorem{definition}{Definition}
\newtheorem{remark}{Remark}

\newcommand{\ve}{\varepsilon}
\newcommand{\sk}{s_{k,\theta_0}^{\theta_1,\theta_2}}
\newcommand{\skinv}{s_{-k,\theta_0}^{\theta_2,\theta_1}}

\newcommand{\rev}[1]{\textcolor{red}{#1}}

\title{Analytic Rigidity and Symbolic Dynamics for Two-Centre Billiards}
\author{Stefano Baranzini \and Susanna Terracini}

\date{}


\begin{document}
		\begin{abstract}
		
We establish a sharp rigidity--chaos dichotomy for planar two-centre billiards, motivated by a natural analogue of the Birkhoff--Poritsky conjecture: the only tables integrable at every energy should be ellipses confocal with the two centres. Let $\Omega$ be a bounded domain with $\mathcal C^1$ boundary containing the segment joining the centres. At every fixed energy $h\geq 0$, if $\partial\Omega$ is not a confocal ellipse, we construct billiard trajectories that shadow the stable and unstable manifolds of the collision--reflection orbit and realise arbitrarily prescribed sequences of sufficiently large winding numbers around the segment. This yields an invariant set semiconjugate to the full shift on a countable alphabet, periodic trajectories with prescribed finite itineraries, and compact invariant subsystems with arbitrarily large topological entropy. If, in addition, $\partial\Omega$ is real-analytic, every real-analytic function on the fixed-energy phase space $M_h$ that is invariant under the billiard map is constant. This establishes the real-analytic form of the two-centre Birkhoff--Poritsky conjecture throughout the non-negative-energy regime.
\end{abstract}

	\maketitle

	\providecommand{\rev}[1]{{\color{red}#1}}

\section{Introduction}

In a Birkhoff billiard, a point particle moves freely inside a bounded, strictly
convex planar domain $\Omega$ and is reflected elastically whenever it reaches
the boundary. Despite the elementary nature of the model, the resulting
dynamical system is extremely rich and can display the full range of behaviours of conservative dynamics, from
complete integrability to positive topological entropy; see
\cite{Birkhoff1927,KozlovTreshchev1991,Tabachnikov2005}
for a general introduction.

Rigidity problems lie at the interface between dynamics and geometry. Rather than merely constructing integrable examples, they ask whether integrability forces the underlying system to have a highly constrained geometric form. Billiards provide a particularly transparent setting for this principle, since their dynamics is encoded by the shape of the reflecting boundary. The classification of integrable billiards can therefore be viewed as an inverse problem: to what extent does an integrable structure determine the geometry of the table?

Elliptic tables play a distinguished role. The billiard inside an ellipse is
integrable: the phase space is foliated by the invariant curves associated with
the family of confocal conics (the caustics), and the corresponding periodic
orbits obey Poncelet's closure theorem \cite{DragovicRadnovic2011}. Whether this
is the only instance of integrability is the content of a longstanding
question, usually attributed to Birkhoff and first recorded in print by Poritsky
\cite{Poritsky1950}.

\begin{conjecture}[Birkhoff--Poritsky]\label{conj:BP}
Among bounded, strictly convex planar domains with $\mathcal{
C}^2$ boundary, the only
integrable billiard tables are the ellipses.
\end{conjecture}

The model admits a natural extension within the framework of \emph{mechanical billiards}: between successive impacts, the particle evolves under a conservative force field, while reflections at $\partial\Omega$ remain elastic. As in Birkhoff billiards, the dynamics is closely intertwined with the geometry of the boundary, although the relevant geometry is no longer Euclidean but is instead determined by the underlying mechanical system. In particular, a first integral of the mechanical flow extends to a first integral of the billiard dynamics only if it is preserved by reflection at every boundary point. This imposes a strong compatibility between the reflecting wall and the integrable structure of the interior dynamics.

This point of view has been exploited to build rich families of integrable billiards starting from the Kepler potential.  Panov
\cite{Panov1994} observed that the billiard in an ellipse with a Newtonian
centre at one of its foci is integrable, at every energy level. A systematic
explanation, and a substantial list of further examples, was later obtained by Kozlov \cite{Kozlov1995} and
Takeuchi and Zhao
\cite{TakeuchiZhao2024conformal,TakeuchiZhao2024spaceforms}. 

In particular, \cite{TakeuchiZhao2024conformal}  raises the Keplerian counterpart of
Conjecture~\ref{conj:BP}, which has been revisited and partially solved in
\cite{BaranziniBarutelloDeBlasiTerracini2025}, except for a possible special position of the centre of attraction (\emph{focal point of the second kind}).

\begin{conjecture}[Birkhoff--Poritsky for Kepler billiards]\label{conj:kepler}
Among bounded, strictly convex planar domains with $\mathcal{C}^2$--boundary, the
only Kepler billiards which are integrable at all energy levels are the ellipses
with the centre of attraction located at one of the foci.
\end{conjecture}

In this paper we study the \emph{two-centre} mechanical billiard problem, in which a particle moves in the plane and is subject to the attraction of two fixed Newtonian centres. 
As one of the simplest genuinely non-central Liouville-integrable systems, the two-centre problem offers a particularly sharp setting in which to investigate rigidity phenomena in integrable dynamics:  the guiding question is which reflecting boundaries are compatible with the conservation laws of the integrable Hamiltonian flow. 

The particle satisfies the following equation of motion:
\begin{equation}\label{eq:2centres}
 \ddot x=\nabla U(x)\;\qquad U(x)=\frac{m_{1}}{|x-c_{1}|}+\frac{m_{2}}{|x-c_{2}|},
  \qquad m_{1}\ge m_{2}>0
\end{equation}
which are associated to the Hamiltonian
\begin{equation}\label{eq:H}
  H(p,x)=\frac{1}{2}|p|^{2}-\frac{m_{1}}{|x-c_{1}|}-\frac{m_{2}}{|x-c_{2}|}.
\end{equation}

The two-centre problem was first solved by Euler in two foundational
memoirs \cite{Euler1766,Euler1767}: it is
Liouville integrable, and it separates in the elliptic--hyperbolic coordinates
 adapted to the confocal family of conics generated by $c_{1}$ and $c_{2}$,
see \cite{Jacobi1884,WaalkensDullinRichter2004} for the classical theory and
for the global structure of the Liouville foliation and Section \ref{sec:two_centres}. Confining the system to a
bounded table and adding an elastic reflection law produces the object of our
study, the \emph{two-centre billiard}. 

Once again, integrability survives in one
distinguished configuration:
if $\partial\Omega$ is the ellipse with foci
exactly at $c_{1}$ and $c_{2}$, the elastic reflection at $\partial \Omega$ is compatible with the separation in elliptic--hyperbolic coordinates, so the first integral of the continuous system is preserved across collisions with the boundary and the billiard remains integrable, see \cite{TakeuchiZhao2024conformal}. Our aim is to show that this configuration is rigid.

\begin{conjecture}\label{conj:main}
Let $\Omega\subset\mathbb{R}^{2}$ be a bounded strictly convex domain containing
the two centres $c_{1},c_{2}$. Then the two-centre billiard in $\Omega$ is
integrable at every energy level if and only if $\partial\Omega$ is an ellipse with foci $c_{1}$ and
$c_{2}$.
\end{conjecture}

We deliberately leave the notion of integrability unspecified: the conjecture is
meant to hold for any of the notions in use for billiards, from the existence of
a non-constant first integral with some prescribed regularity to the existence of a foliation of the domain of the billiard map by invariant curves. As a matter of fact, the literature on billiards contains several non-equivalent notions of
integrability, which give  correspondingly different
meanings to all the  conjectures under consideration. 

\subsection*{Main Results}

We prove the 
analytic rigidity predicted by
Conjecture~\ref{conj:main} \emph{at any non negative energy level} under the sole  assumption
that the bounded domain contains the whole segment
$[c_1,c_2]$ and has $\mathcal C^1$ boundary.  If $\Omega$ is not a confocal ellipse,
analytic integrability does not merely fail: the billiard contains a symbolic
subsystem carrying the full shift on a countable alphabet.

We have the following results.

\begin{thm}\label{thm:A}
Let $\Omega\subset\mathbb{R}^{2}$ be a bounded domain with $\mathcal{C}^{1}$
boundary containing the segment $[c_{1},c_{2}]$. Assume further that
$\partial\Omega$ is \emph{not} an ellipse with foci $c_{1}$ and $c_{2}$. Then,
for every energy $h\ge 0$, the two-centre billiard map in $\Omega$ at energy $h$
admits an invariant set $\Lambda_{h}$ on which it is semiconjugate to the full
shift on a countable alphabet. Thus, it contains compact invariant subsystems of  arbitrarily large topological entropy. 
\end{thm}

\begin{thm}\label{thm:B}
Under the assumptions of Theorem~\ref{thm:A} and, additionally, that the boundary $\partial\Omega$ is real analytic, for any energy $h\ge0$ every real-analytic first integral of the billiard map  is constant. 
\end{thm}

Some comments are in order. First of all, let us remark that the dichotomy integrable/non-integrable is sharp and requires no genericity. The only
domains that do not support a symbolic dynamics are confocal ellipses themselves, on which the billiard
is integrable. 

Furthermore, the results hold at any non-negative energy level, in contrast with the perturbative results of 
\cite{BaranziniBarutelloDeBlasiTerracini2025,Baranzini2025}  which hold for $h$ large, and require minimal regularity assumptions on $\Omega$.

The mechanism behind our theorems is the following. At energy $h$, the
segment $[c_{1},c_{2}]$ supports a collision--reflection periodic orbit whose
stable and unstable manifolds sweep the region outside the segment.
Trajectories winding many times around $[c_{1},c_{2}]$ shadow long arcs of
these manifolds. The connection with the boundary geometry is encoded by
\begin{equation}\label{eq:f-intro}
    f(\theta)=\tfrac{1}{2}\bigl(
    |\gamma(\theta)-c_{1}|+|\gamma(\theta)-c_{2}|
    \bigr).
\end{equation}
Indeed, at $\gamma(\theta)$, the limiting incoming and outgoing branches
satisfy the elastic reflection law precisely when $f'(\theta)=0$, or
equivalently when $\partial\Omega$ is tangent there to a confocal ellipse.
Thus, for large winding numbers, the reflection equations are small
perturbations of $f'=0$; their solutions close the chains, while the freedom
in choosing the winding numbers generates the symbolic dynamics. Finally,
$f$ is constant if and only if $\partial\Omega$ is itself a confocal ellipse,
precisely the case in which this construction fails.

The obstruction to analytic integrability is localised near the stable
directions over the stable critical points of  $f$. The existence of this symbolic
subsystem does not by itself imply global chaotic behaviour, such as
ergodicity. Indeed, under suitable regularity and convexity assumptions,
Section~\ref{subsec:lazutkin} adapts the argument of Lazutkin, in the form
developed in \cite{DiasCarneiroEtAl2024}, to prove the existence of invariant
curves near the boundary.

\subsection*{Around the Birkhoff Conjecture}

Conjecture~\ref{conj:BP} remains open in full generality, despite substantial progress over the past few decades. Under the strong assumption that the entire phase cylinder is foliated by non-contractible invariant curves, Bialy's Hopf-type argument~ \cite{Bialy1993}
implies that the table must be a circle.
From a perturbative perspective, local versions were proved by Avila, De~Simoi and Kaloshin
\cite{AvilaDeSimoiKaloshin2016} for deformations of ellipses of small
eccentricity, extended by Kaloshin and Sorrentino
\cite{KaloshinSorrentino2018} to arbitrary eccentricity, and by Huang, Kaloshin
and Sorrentino \cite{HuangKaloshinSorrentino2018} and Koval \cite{Koval2026} for foliations arbitrarily close to the boundary; see \cite{KaloshinSorrentino2018survey} for a
survey. For centrally symmetric tables, Bialy and Mironov
\cite{BialyMironov2022} resolved the conjecture under the corresponding
near-boundary integrability assumption.

A different approach consists in exhibiting transverse homoclinic
intersections --- a splitting of separatrices --- which directly rule out the
existence of an analytic first integral, see for instance 
\cite{DelshamsRamirezRos1996,BaldomaFlorioLeguilSeara2026}. This is closer to the philosophy of the present contribution, even if the mechanism is slightly different: instead of proving a
	transverse homoclinic intersection, we construct a symbolic subsystem directly
	from arcs shadowing the invariant manifolds of a periodic orbit contained in $\Omega$.

Thus, for non-negative energies and within the class of real-analytic first
integrals, our results establish the corresponding version of
Conjecture~\ref{conj:main}.

\subsection*{Order and chaos in Kepler-type billiards}
Interest in mechanical billiards goes back at least to Boltzmann
\cite{Boltzmann1868}, who considered a particle subject to a central force, a superposition of Kepler potential and a term of order $\beta/r^2$,
bouncing on a straight wall as a candidate model for ergodic behaviour. This has recently attracted some interest since, contrary
to Boltzmann's claim, the model is quasi-integrable when $\beta$ is small, as clarified only recently in
\cite{Felder2021}. 
The reason is that the corresponding Kepler billiard, for $\beta = 0$, is integrable. On the other hand, it is chaotic for positive energies when $\beta>0$ (\cite{DeBlasiTakeuchiTerracini}).

The Kepler and two-centre billiards in an ellipse are part of a broader class of integrable mechanical billiards. Kozlov \cite{Kozlov1995} embedded the two-centre model into a larger family of separable systems for which the billiard inside a confocal ellipse remains integrable. The common mechanism is the existence of a first integral of the interior Hamiltonian flow that is also preserved by elastic reflection at the boundary. The same principle underlies Panov's theorem and the integrability of the Kepler billiard with a straight wall; in the latter case, the preserved quantity combines the angular momentum with components of the Laplace--Runge--Lenz vector. We refer to \cite{JaudZhao2024} for a detailed description of its dynamics. More recently, this principle has been used in \cite{TakeuchiZhao2024projective} to construct further integrable mechanical billiards by means of conformal transformations of the plane.

Chaotic features of Kepler and refraction billiards have been investigated,
both numerically and rigorously, in
\cite{BarutelloDeBlasiTerracini2023,BarutelloCherubiniDeBlasi2025}.
The first substantial step towards Conjecture~\ref{conj:kepler} was taken in
\cite{BaranziniBarutelloDeBlasiTerracini2025}, where it is proved that for a
strictly convex domain $\Omega$ with real-analytic boundary the Kepler billiard
is not analytically integrable at high energies unless one of the following two alternatives occurs. Either $\partial\Omega$ is an
ellipse with the centre at one of its foci, or $\Omega$ has a very special invariant curve of rotation number $1/2$ for the corresponding Birkhoff billiard. It consists of the pencil of lines through $c$ and their reflections. Such
points cannot be excluded and they occur on an infinite-dimensional family of
non-elliptic domains, see also \cite{Baranzini2026}.
The zero-energy case of Conjecture~\ref{conj:kepler} is discussed in \cite{Zhao2025}, where a parallel with the Birkhoff billiard is drawn, allowing to translate the main result of \cite{BialyMironov2022} to the Kepler setting.

Another feature for which the two-centre problem is special, is that the $N$-centre problem for $N\ge 3$ is no longer integrable. In the absence of a reflecting wall, non-integrability
and symbolic dynamics have been established under several sets of hypotheses,
depending on the energy regime and on the precise form of the potential; see
\cite{Bolotin1984,KleinKnauf1992,BolotinNegrini2001,Knauf2002,%
	KnaufTaimanov2005,SoaveTerracini2012,BarutelloCanneoriTerracini2021,%
	BaranziniCanneori2024}.
The case of a
straight reflecting wall and  $N$-centre with $N\ge 2$ is treated in \cite{Baranzini2025}.

\subsection*{Strategy of the proof}

At a fixed energy, trajectories of the two-centre problem are, up to a time
reparametrisation, geodesics of the Jacobi--Maupertuis metric. On the
regularising double cover this metric is smooth, complete, and negatively
curved; hence a geodesic joining two prescribed endpoints is uniquely
determined by its homotopy class. 

We use them to construct minimising arcs between boundary points with a
prescribed, large winding number around $[c_1,c_2]$. This yields a countable family of generating functions. As the winding number
tends to infinity, the arcs shadow the stable and unstable manifolds of the
collision--reflection orbit. When their endpoints lie in a sufficiently small
boundary interval near a minimum of $f$, for large winding numbers these arcs remain inside $\Omega$ and
meet $\partial\Omega$ only at their endpoints.

Concatenating these arcs reduces the reflection law at each intermediate
impact to a stationarity equation for the total Jacobi--Maupertuis length.
For large winding numbers, these equations are uniformly approximated by
$f'=0$. The finite-dimensional Poincar\'e--Miranda theorem yields periodic
orbits with prescribed winding data, while its countable-product version
produces complete trajectories realising arbitrary bi-infinite sequences.
Finite subalphabets then give compact invariant sets with arbitrarily large
topological entropy.

For analytic non-integrability, we fix the first impact and use one-sided
sequences with increasing winding numbers. The resulting initial conditions
accumulate on the stable graph and force any first integral to be constant on
a non-empty open subset of $M_h$. When $\partial\Omega$ is real-analytic,
$M_h$ is a connected real-analytic manifold, so the analytic identity
principle extends this constancy to all of $M_h$.

\subsection*{Structure of the paper}

Section~\ref{sec:two_centres} collects some classical facts on the
two-centre problem in elliptic coordinates, regularises the collisions, and
introduces the collision--reflection orbit together with the vector fields tangent to the stable and unstable manifolds.  Section~\ref{sec:billiard_map} introduces the billiard map, discusses return and grazing, and records a
Jacobi--Maupertuis convexity criterion for the global well posedness of the problem (together with the near-boundary KAM invariant 
curves under stronger regularity).  Section~\ref{sec:approximation} constructs
the minimising arcs, proves their approximation to the invariant manifolds, and
establishes the confinement needed for the symbolic construction.  Finally,
Section~\ref{sec:symbolic} applies the Poincar\'e--Miranda theorem to build the symbolic
dynamics and proves analytic non-integrability.

\subsection*{Acknowledgements}
The authors acknowledge the support of the INdAM--GNAMPA Research Group. The first author also acknowledges financial support from the Alexander von Humboldt Foundation.

	\section{The 2-centre problem}
\label{sec:two_centres}

In this section we recall some of the features of the
$2$-centre problem that will be used in the sequel. The separation of the problem
in elliptic-hyperbolic coordinates goes back to Euler \cite{Euler1766,Euler1767}; we also
record the Jacobi--Maupertuis formulation and the curvature computation needed
later. The elliptic-hyperbolic coordinates are also well suited to regularisation: on the branched
double cover the collisions become regular points, while the motion supported on
the segment $[c_1,c_2]$ gives rise to the hyperbolic collision--reflection orbit
around which the later construction is organised.

We write $r_i=|x-c_i|$ and $\hat r_i=(x-c_i)/r_i$, and consider the
Hamiltonian
\begin{equation}
  \label{eq:hamiltonian_2centres}
  H(p,x)=\frac12\vert p\vert^2-U(x),
  \qquad
  U(x):=\frac{m_1}{\vert x-c_1\vert}+\frac{m_2}{\vert x-c_2\vert}
       =\frac{m_1}{r_1}+\frac{m_2}{r_2},
\end{equation}
with $m_1\ge m_2>0$. We shall also write $\mu_1=m_1+m_2$ and
$\mu_2=m_1-m_2\ge0$, so that $\mu_1-\mu_2=2m_2>0$. 

In the following discussion we place the centres at
\[
c_1=-e_1=(-1,0),\qquad c_2=e_1=(1,0),\qquad \ell:=|c_1-c_2|=2.
\]
This greatly simplifies all the formulae. Note that, thanks to the invariance under isometries and to the homogeneity of the potential $U$, there is no loss of generality in doing so.

\subsection{Elliptic-hyperbolic coordinates}

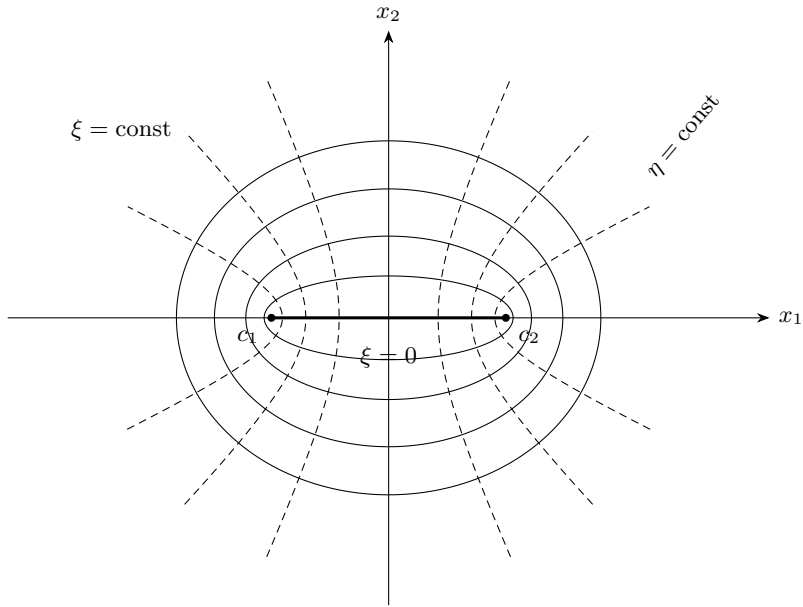
\begin{figure}
\centering
\begin{tikzpicture}[x=1.55cm,y=1.55cm,>=Stealth,
  every node/.style={font=\small}]


\draw[->,thin] (-3.25,0) -- (3.25,0) node[right] {$x_1$};
\draw[->,thin] (0,-2.45) -- (0,2.45) node[above] {$x_2$};

\draw[line width=1.15pt] (-1,0)--(1,0);
\fill (-1,0) circle (1.5pt) node[below left=1pt] {$c_1$};
\fill ( 1,0) circle (1.5pt) node[below right=1pt] {$c_2$};
\node[below=7pt] at (0,0) {$\xi=0$};

\foreach \xi in {0.35,0.65,0.95,1.20} {
  \draw[thin]
    plot[domain=0:360,samples=220,variable=\t]
    ({cosh(\xi)*cos(\t)},{sinh(\xi)*sin(\t)});
}

\foreach \eta in {25,45,65,115,135,155} {
  \draw[thin,densely dashed]
    plot[domain=-1.55:1.55,samples=180,variable=\s]
    ({cosh(\s)*cos(\eta)},{sinh(\s)*sin(\eta)});
}

\node[fill=white,inner sep=1.5pt] at (-2.28,1.60)
  {$\xi=\mathrm{const}$};
\node[fill=white,inner sep=1.5pt,rotate=50] at (2.52,1.55)
  {$\eta=\mathrm{const}$};

\end{tikzpicture}

\caption{Elliptic-hyperbolic coordinates associated with the two centres \(c_1\) and \(c_2\). The curves \({\xi=\mathrm{const}}\) are ellipses confocal with \(c_1\) and \(c_2\), whereas the curves \({\eta=\mathrm{const}}\) are the branches of the corresponding confocal hyperbolas. The degenerate level \({\xi=0}\) is a double cover of the segment $[c_1,c_2]$.}
\label{figure:elliptic_hyperbolic}
\end{figure}

First of all, let us introduce the elliptic-hyperbolic coordinates.
We consider the change of coordinates
\begin{equation}
  \label{eq:elliptic_coordinates}
  x=(x_1,x_2)=\Phi(\xi,\eta)=\bigl(\cosh\xi\cos\eta,\ \sinh\xi\sin\eta\bigr),
\end{equation}
or equivalently, in complex notation $z=x_1+ix_2$ and $w=\xi+i\eta$,
\begin{equation}
  \label{eq:cosh}
  z=\cosh w .
\end{equation}
A straightforward computation shows that
\[
  \begin{cases}
    dx_1=\sinh\xi\cos\eta\,d\xi-\cosh\xi\sin\eta\,d\eta,\\
    dx_2=\cosh\xi\sin\eta\,d\xi+\sinh\xi\cos\eta\,d\eta,
  \end{cases}
\]
and thus the Euclidean metric in the new coordinates reads
\begin{align}
  \label{eq:euclidean_elliptic}
  dx_1^2+dx_2^2
   &=\bigl(\sinh^2\xi\cos^2\eta+\cosh^2\xi\sin^2\eta\bigr)
     \bigl(d\xi^2+d\eta^2\bigr)\\
  \nonumber
   &=\bigl(\cosh^2\xi-\cos^2\eta\bigr)\bigl(d\xi^2+d\eta^2\bigr).
\end{align}
Moreover, the distances $\vert x\pm e_1\vert$ are given by
\begin{align*}
  \vert x\pm e_1\vert
   &=\sqrt{\cosh^2\xi\cos^2\eta+1\pm2\cosh\xi\cos\eta+\sinh^2\xi\sin^2\eta}\\
   &=\sqrt{\sinh^2\xi+1+\cos^2\eta\pm2\cosh\xi\cos\eta}
    =\cosh\xi\pm\cos\eta,
\end{align*}
that is
\begin{equation}
  \label{eq:r1r2}
  r_1=\cosh\xi+\cos\eta,\qquad r_2=\cosh\xi-\cos\eta.
\end{equation}
It follows that the level sets of the first coordinate correspond to ellipses with foci at $\pm e_1$.
Indeed
\begin{equation}
  \label{eq:function_ellipse}
  2\cosh\xi=\vert x+e_1\vert+\vert x-e_1\vert=r_1+r_2 .
\end{equation}
 A similar consideration, which explains the name \textit{elliptic-hyperbolic}, holds for the level sets of $\eta$: they are branches of confocal hyperbolas.
Two consequences of \eqref{eq:r1r2} will be used repeatedly. First, comparing
with \eqref{eq:euclidean_elliptic},
\begin{equation}
  \label{eq:conformal_factor}
  \cosh^2\xi-\cos^2\eta=r_1r_2 ,
\end{equation}
so the conformal factor vanishes exactly at the two collisions. Second, the potential $U$ satisfies
\begin{equation}
  \label{eq:lambdaU}
  r_1r_2\,U=m_1r_2+m_2r_1=\mu_1\cosh\xi-\mu_2\cos\eta ,
\end{equation}
which extends to a real-analytic function on the whole $(\xi,\eta)$ plane and is bounded below by $\mu_1-\mu_2=2m_2>0$. 

Finally, let us note that the map $\Phi$ of
\eqref{eq:elliptic_coordinates} descends to a two-to-one conformal branched
covering
\[
  \Phi\colon\ \mathcal{C}:=\mathbb{R}\times\mathbb{S}^1\ \longrightarrow\ \mathbb{R}^2 ,
\]
whose deck transformation is the involution
\begin{equation}
  \label{eq:deck}
  \mathfrak{d}(\xi,\eta)=(-\xi,-\eta)
\end{equation}
induced by $w\mapsto-w$. Its branch points are
\[
  (\xi,\eta)=(0,\pi)\ \longmapsto\ -e_1,
  \qquad
  (\xi,\eta)=(0,0)\ \longmapsto\ e_1 ,
\]
the two fixed points of $\mathfrak{d}$. Moreover, the circle $\{\xi=0\}$ is a double cover
of the segment $[-e_1,e_1]$.

\subsection{Separation of the equation of motion}

Elliptic-hyperbolic coordinates are very useful since they allow us to separate the variables in the Hamiltonian given in \eqref{eq:hamiltonian_2centres}. In the $(\xi,\eta)$ coordinates it reads
\begin{align}
  \label{eq:hamiltonian_elliptic_hyperbolic}
  \tilde H(p_\xi,p_\eta,\xi,\eta)
   &=\frac{p_\xi^2+p_\eta^2}{2(\cosh^2\xi-\cos^2\eta)}
     -\frac{m_1}{\cosh\xi+\cos\eta}-\frac{m_2}{\cosh\xi-\cos\eta}\\
  \nonumber
   &=\frac{\frac{p_\xi^2+p_\eta^2}{2}
           -\mu_1\cosh\xi+\mu_2\cos\eta}
          {\cosh^2\xi-\cos^2\eta},
\end{align}
where the second equality uses \eqref{eq:conformal_factor} and
\eqref{eq:lambdaU}. On the energy level $\{\tilde H=h\}$, the flow of $\tilde H$ is a time
re-parametrisation of the zero-energy flow of the Hamiltonian
\begin{equation}
  \label{eq:def_K}
  K(p_\xi,p_\eta,\xi,\eta)=\frac{p_\xi^2+p_\eta^2}{2}
    -\mu_1\cosh\xi+\mu_2\cos\eta
    -h\bigl(\cosh^2\xi-\cos^2\eta\bigr),
\end{equation}
and $K$ splits as a sum of two Hamiltonians depending solely on $\xi$ and on
$\eta$ respectively,
\begin{equation}
  \label{eq:def_K1K2}
  K_1(p_\xi,\xi)=\frac12p_\xi^2-W_1(\xi),
  \qquad
  K_2(p_\eta,\eta)=\frac12p_\eta^2+W_2(\eta),
\end{equation}
where
\begin{equation}
  \label{eq:def_W1W2}
  W_1(\xi)=\bigl(\mu_1+h\cosh\xi\bigr)\cosh\xi,
  \qquad
  W_2(\eta)=\bigl(\mu_2+h\cos\eta\bigr)\cos\eta .
\end{equation}
Away from the branch points, the level $\{K=0\}$ corresponds exactly to
the energy shell $\{\tilde H=h\}$ and provides its regularised extension across
the collisions. On $\{K=0\}$, the common value
\begin{equation}
  \label{eq:def_K0}
  K_0:=K_1=-K_2
\end{equation}
is the separation constant, that is, Euler's first integral restricted to
$\{\tilde H=h\}$. We shall repeatedly use the identity
\begin{equation}
  \label{eq:speed_on_shell}
  \tfrac12\bigl(p_\xi^2+p_\eta^2\bigr)=W_1(\xi)-W_2(\eta)=r_1r_2\,(h+U)
  \qquad\text{on }\{K=0\},
\end{equation}
which follows from \eqref{eq:def_K1K2}, \eqref{eq:conformal_factor} and
\eqref{eq:lambdaU}, and whose right-hand side is bounded below by $2m_2>0$.

Let us analyse the regularised dynamics.
The first Hamiltonian has a unique critical point at the origin. Indeed
$\partial_\xi K_1=-\sinh\xi\,(\mu_1+2h\cosh\xi)$ and $\mu_1+2h\cosh\xi>0$, so
$\partial_\xi K_1$ vanishes only for $\xi=0$. A direct computation shows that this critical point is a saddle: in the
coordinates $(p_\xi,\xi)$ the linearization of the field is
\begin{equation*}
  JD^2K_1(0,0)=\begin{pmatrix}0&-1\\1&0\end{pmatrix} \begin{pmatrix}1&0\\0&-(\mu_1+2h)\end{pmatrix} = \begin{pmatrix}0&\mu_1+2h\\1&0\end{pmatrix} 
\end{equation*}
and  has eigenvalues
$\pm\sqrt{\mu_1+2h}$.

For completeness, we also record the critical points of the second Hamiltonian,
although their detailed phase portrait will not be used below. Since
\[
  \partial_\eta K_2=-\sin\eta\,(\mu_2+2h\cos\eta),
\]
the following alternatives occur.
\begin{itemize}
  \item If $0\le h<\mu_2/2$ (hence $\mu_2>0$), the only critical points are
  $(p_\eta,\eta)=(0,\zeta_k)$, where $\zeta_k=\pi k$. They are saddles for even
  $k$ and minima for odd $k$.
  \item If $h=\mu_2/2>0$, the same points occur; the even ones are saddles,
  whereas the odd ones are degenerate minima.
  \item If $h>\mu_2/2$, besides $\zeta_k=\pi k$ there are the critical points
  \[
    \chi_k^\pm=\pm\arccos\!\left(-\frac{\mu_2}{2h}\right)+2\pi k,
    \qquad k\in\mathbb Z.
  \]
  The points $\zeta_k$ are saddles and the points $\chi_k^\pm$ are minima; moreover
  $\zeta_{2k}<\chi_k^+<\zeta_{2k+1}<\chi_{k+1}^-$.
\end{itemize}
In the exceptional case $\mu_2=h=0$, one has $K_2=p_\eta^2/2$, so the whole
circle $\{p_\eta=0\}$ consists of degenerate critical points.

\begin{lemma}
  \label{lemma:monotonicity_xi}
  Any solution $x(t)$ of the $2$-centre problem at energy $h\ge0$ crosses the
  segment $[-e_1,e_1]$ at most once. More precisely, let $K_0$ be its separation
  constant \eqref{eq:def_K0} and set $\bar K_0:=-W_1(0)=-(\mu_1+h)=K_1(0,0)$.
  Then exactly one of the following holds.
  \begin{enumerate}
    \item[(i)] If $K_0>\bar K_0$, then $\xi(t)$ is strictly monotone and vanishes
    exactly once: the solution crosses $(-e_1,e_1)$ once, and $f(x(t))$ has a
    unique critical point there, a strict minimum with $f=1$.
    \item[(ii)] If $K_0<\bar K_0$, then $\xi(t)$ never vanishes and has exactly
    one turning point: the solution does not meet $[-e_1,e_1]$, and $f(x(t))$ has
    a unique critical point.
    \item[(iii)] If $K_0=\bar K_0$, then either $x(t)$ is the
    collision--reflection orbit on $[e_1,-e_1]$, along which $f\equiv1$, or $x(t)$
    lies on one of the branches of the stable or unstable manifold of that orbit,
    along which $f(x(t))$ is strictly monotone and tends to $1$.
  \end{enumerate}
\end{lemma}

\begin{proof}
  By \eqref{eq:def_K1K2} and \eqref{eq:def_W1W2},
  $W_1'(\xi)=\sinh\xi\,(\mu_1+2h\cosh\xi)$, and since $\mu_1+2h\cosh\xi>0$ the
  function $W_1$ is strictly decreasing on $(-\infty,0)$, strictly increasing on
  $(0,+\infty)$, attains its minimum $W_1(0)=-\bar K_0$ only at $\xi=0$, and
  tends to $+\infty$ as $\vert\xi\vert\to\infty$. On $\{K_1=K_0\}$ we have
  \begin{equation}
    \label{eq:pxi_squared}
    p_\xi^2=2\bigl(K_0+W_1(\xi)\bigr).
  \end{equation}
  Thus, the possible zeros of $p_\xi$ on the level $\{K_1=K_0\}$ occur at
  the solutions of $W_1(\xi)=-K_0$. This equation has, respectively, no
  solution, two solutions, or the single solution $\xi=0$, according as
  $K_0>\bar K_0$, $K_0<\bar K_0$, or $K_0=\bar K_0$.

  If $K_0>\bar K_0$ the right-hand side of \eqref{eq:pxi_squared} is bounded
  below by $2(K_0-\bar K_0)>0$, so $p_\xi$ has constant sign, and $\xi$ is strictly monotone along the
  maximal orbit and vanishes exactly once. If $K_0<\bar K_0$ the two
  turning points are $\pm\xi_*$ with $\xi_*>0$, lying in the two distinct
  components of $\{W_1\ge-K_0\}$; each orbit is confined to one of them and has a
  single turning point. If $K_0=\bar K_0$ the level set consists of the
  equilibrium $(0,0)$ together with its four separatrix branches, along which
  $\xi$ is strictly monotone and tends to $0$ without attaining it.

  The assertions about $f$ follow from \eqref{eq:function_ellipse}: since $\cosh$ is
  even with a strict minimum at the origin, a critical point of
  $f(x(t))=\cosh\xi(t)$ occurs either where $\dot\xi=0$ or where $\xi=0$. The
  last claim of (iii) holds because on the separatrices $\xi\to0$, whence
  $f\to1$.
\end{proof}

We will be mainly working on the open set $\{\xi>0\}\subset \mathcal{C}$, on which the flow of $K$ is smooth, albeit incomplete.

Using the deck involution defined in \eqref{eq:deck}, it is however possible to define a full regularization of the 2-centre problem. Let us observe that $d \mathfrak{d}$  acts \emph{freely} on
$\{K=0\}\subset T^*\mathcal{C}$ for any $h\ge0$.
Indeed, a fixed point of $d{\mathfrak{d}}$ must have $p_\xi=p_\eta=0$ and lie  over a branch point. But by \eqref{eq:speed_on_shell} one has
$\tfrac12(p_\xi^2+p_\eta^2)=W_1-W_2$, which equals $2m_2$ at $(0,0)$ and
$2m_1$ at $(0,\pi)$, hence is strictly positive.

We can thus give the following definition of \textit{regularized phase space and flow}.
\begin{definition}[Regularised phase space and flow]
	\label{def:Psi}
	For $h\ge0$ set
	\begin{equation}
		\label{eq:def_Sigma_hat}
		\widehat\Sigma_h:=\{K=0\}\big/{d \mathfrak{d}} .
	\end{equation}
\end{definition}

By \eqref{eq:speed_on_shell}, $(p_\xi,p_\eta)\neq(0,0)$ on $\{K=0\}$,
so $\{K=0\}$ is a real-analytic three-dimensional hypersurface on which $d\mathfrak{d}$ acts freely, so $\widehat\Sigma_h$ is a
real-analytic three-dimensional manifold and the definition is well posed.

We conclude this section with some remarks. First of all, away from the centres, the flow of $K$, its projection onto $\widehat \Sigma_h$ and the original 2-centre flow are essentially equivalent. Indeed, up to the time change
\begin{equation}
	\label{eq:time_change}
	dt_{\mathrm{phys}}=r_1r_2\,dt 
\end{equation}
and, possibly up to an identification under $\mathfrak{d}$, the solutions to these systems are all the same.

We will always refer to the time in the $(\xi,\eta)$ coordinates as \textit{regularised time} and to that of the original Hamiltonian system as \textit{physical time}.

The space $\widehat \Sigma_h$ is essentially a blow-up of the energy surface of $H$ as given in \eqref{eq:H}.
What is added as fiber over  $\pm e_1$ is a circle
$\mathcal{Z}_\pm$.

\subsection{The collision--reflection orbit and the fields $X_\pm^h$}

Throughout this subsection $h\ge0$ is fixed, and we work with the Hamiltonian flow of $K$ given in \eqref{eq:def_K} for which $\dot\xi=p_\xi$ and
$\dot\eta=p_\eta$ on the cylinder $\mathcal{C}$. We begin by isolating the periodic orbits around
which the whole construction is built.

\begin{prop}
  \label{prop:collision_orbit}
  For every $h\ge0$ the Hamiltonian $K_1$ has a unique equilibrium,
  $(p_\xi,\xi)=(0,0)$, and it is a hyperbolic saddle with eigenvalues
  $\pm\lambda$, where
  \begin{equation}
    \label{eq:def_lambda}
    \lambda:=\sqrt{\mu_1+2h}.
  \end{equation}
   It corresponds to two orbits $\Gamma_h$ and $\Gamma_h^{-1}$, the only along which $\xi$ is
  constant. $\Gamma_h$ and $\Gamma^{-1}_h(t)=\Gamma_h(-t)$  correspond to generalized solutions of  \eqref{eq:2centres} that span the segment $[-e_1,e_1]$. Moreover, their stable and unstable manifolds satisfy
  \begin{equation}
    \label{eq:manifolds_level_set}
    W^s(\Gamma_h)\cup W^u(\Gamma_h)\cup     W^s(\Gamma_h^{-1})\cup W^u(\Gamma_h^{-1})
    =\left\{K_0=\bar K_0\right\},
    \qquad \bar K_0=-(\mu_1+h).
  \end{equation}
  In particular, through every
  $x\in\mathbb{R}^2\setminus[-e_1,e_1]$ there pass exactly two orbits asymptotic
  to $\Gamma_h$ and $\Gamma_h^{-1}$ as $t\to+\infty$, and exactly two orbits asymptotic to $\Gamma_h$ and $\Gamma_h^{-1}$  as
  $t\to-\infty$.
\end{prop}

\begin{proof}
  Uniqueness of the equilibrium and the hyperbolicity of the linearization of $K$ were
  already discussed in Section \ref{sec:two_centres}.

  The segment $[-e_1,e_1]$ lifts to the circle $\{0\}\times \mathbb{S}^1$ and conversely, if  $\xi$ is constant, $p_\xi=0$ and $\dot p_\xi=W_1'(\xi)=0$. Since $W_1'$ vanishes only at
  $\xi=0$, the two conditions uniquely identify $\Gamma_h$ and $\Gamma_h^{-1}$.
  
   Identity
  \eqref{eq:manifolds_level_set} holds because $K_0$ is a first integral taking
  the value $\bar K_0=K_1(0,0)$ on $\Gamma_h$, and because on that level the
  $(\xi,p_\xi)$ dynamics consists of the equilibrium together with the four
  separatrix branches of the saddle.

  Finally, fix $x\in\mathbb{R}^2\setminus[-e_1,e_1]$ and let $(\xi,\eta)$ with
  $\xi>0$ be the lift under $\Phi$. On the level $K_0=\bar K_0$, the
  $\xi$-momentum has the two signs
  \[
    p_\xi=\pm\sqrt{2\bigl(W_1(\xi)-W_1(0)\bigr)},
  \]
  while $p_\eta$ has the two non-zero signs
  \[
    p_\eta=\pm\sqrt{2\bigl(W_1(0)-W_2(\eta)\bigr)}.
  \]
  The choice $p_\xi<0$ gives the two orbits asymptotic in the future and the choice
  $p_\xi>0$ the two orbits asymptotic in the past; the sign of $p_\eta$ distinguishes the
  two orientations of $\Gamma_h$.
  Note that the $\xi$ coordinate of these solutions never vanishes. Thus,  if we had instead chosen the lift with $\xi<0$, we would have found lifts with $\xi<0$.
\end{proof}

We can now define two vector fields on $\mathbb{R}^2\setminus[-e_1,e_1]$. Since the elliptic-hyperbolic coordinates given in \eqref{eq:elliptic_coordinates} are one-to-one on the set $\{\xi>0\}\subset \mathcal{C}$, we can work directly on the latter space, with the $(\xi,\eta)$ coordinates. By Proposition~\ref{prop:collision_orbit}, exactly
two orbits through $(\xi,\eta)$ are asymptotic to $\Gamma_h$ in the future. We define
$X^h_+$ and $X^h_-$ to be their regularised coordinate velocities; the sign $\pm$ records the direction of winding. 
Explicitly, solving
$K_1=\bar K_0$ and $K_2=-\bar K_0$ for $p_\xi$ and $p_\eta$, we find 
\begin{equation}
  \label{eq:def_Xpm}
  X^h_\pm=\begin{pmatrix}-a(\xi)\\[2pt] \pm b(\eta)\end{pmatrix},
  \qquad
  \begin{aligned}
    a(\xi)&:=\sqrt{2}\sqrt{W_1(\xi)-(\mu_1+h)}\,,\\
    b(\eta)&:=\sqrt{2}\sqrt{\mu_1+h-W_2(\eta)}\,,
  \end{aligned}
\end{equation}
with $W_1,W_2$ as in \eqref{eq:def_W1W2}.
We remark that, after time reversal, the fields $-X^h_+$ and $-X^h_-$
are proportional to the velocities of the two orbits asymptotic to $\Gamma_h$ in the past.

Two elementary facts about \eqref{eq:def_Xpm} will be used. First, $b$ never
vanishes: writing $u=\cos\eta\in[-1,1]$,
\begin{equation}
  \label{eq:pos_peta}
  b(\eta)^2=2(\mu_1+h-\bigl(\mu_2+hu\bigr)u)
  =2(\mu_1-\mu_2u+h\bigl(1-u^2\bigr))\ \ge\ 2(\mu_1-\mu_2)\ =\ 4m_2\ >\ 0 ,
\end{equation}
so along every leaf of $W^{s}(\Gamma_h)\cup W^{u}(\Gamma_h)$ the coordinate
$\eta$ is strictly monotone, and the number of windings around $[-e_1,e_1]$ is
well defined. Second, $a(\xi)$ vanishes precisely for $\xi=0$, that is,
precisely over the segment $[-e_1,e_1]$. In fact
\begin{equation}
  \label{eq:factorisation}
  a(\xi)^2=2(W_1(\xi)-\mu_1-h)
  =2\bigl(\cosh\xi-1\bigr)\Bigl(h\bigl(\cosh\xi+1\bigr)+\mu_1\Bigr)
\end{equation}
vanishes if and only if
$\cosh\xi=1$, that is $\xi=0$.
\begin{figure}
\centering
\begin{tikzpicture}
\begin{axis}[
    axis equal image,
    xmin=-1.65, xmax=1.65,
    ymin=-1.20, ymax=1.20,
    axis lines=middle,
    xlabel={$x_1$},
    ylabel={$x_2$},
    xtick={-1,0,1},
    ytick=\empty,
    samples=900,
    clip=false,
]

%
%

\def\phase{0.35}

\addplot[
    very thick,
    no marks,
    domain=1:20,
    variable=\t,
]
(
    {cosh(1/\t)*cos(deg(\t+\phase))},
    {sinh(1/\t)*sin(deg(\t+\phase))}
);

\addplot[
    very thick,
    dashed,
    no marks,
    domain=1:20,
    variable=\t,
]
(
    {cosh(1/\t)*cos(deg(-\t-\phase))},
    {sinh(1/\t)*sin(deg(-\t-\phase))}
);

\foreach \ta/\tb in {1.55/1.82,3.10/3.38,5.25/5.55,8.20/8.55,12.00/12.40,16.00/16.45}{
  \addplot[
      very thick,
      -{Stealth[length=3.0mm,width=2.2mm]},
      no marks,
      domain=\ta:\tb,
      variable=\t,
  ]
  (
      {cosh(1/\t)*cos(deg(\t+\phase))},
      {sinh(1/\t)*sin(deg(\t+\phase))}
  );
}

\foreach \ta/\tb in {1.55/1.82,3.10/3.38,5.25/5.55,8.20/8.55,12.00/12.40,16.00/16.45}{
  \addplot[
      very thick,
      dashed,
      -{Stealth[length=3.0mm,width=2.2mm]},
      no marks,
      domain=\ta:\tb,
      variable=\t,
  ]
  (
      {cosh(1/\t)*cos(deg(-\t-\phase))},
      {sinh(1/\t)*sin(deg(-\t-\phase))}
  );
}

\addplot[
    black,
    line width=1.4pt
]
coordinates {(-1,0) (1,0)};

\addplot[
    only marks,
    mark=*,
    mark size=1.6pt
]
coordinates {(-1,0) (1,0)};


\node[anchor=west] at (axis cs:1.07,0.64) {$X_+^h$};
\node[anchor=west] at (axis cs:1.07,-0.64) {$X_-^h$};

\end{axis}
\end{tikzpicture}
\caption{$\Gamma_h$ and the fields $X^h_\pm$ in the $(x_1,x_2)$ plane}
\end{figure}
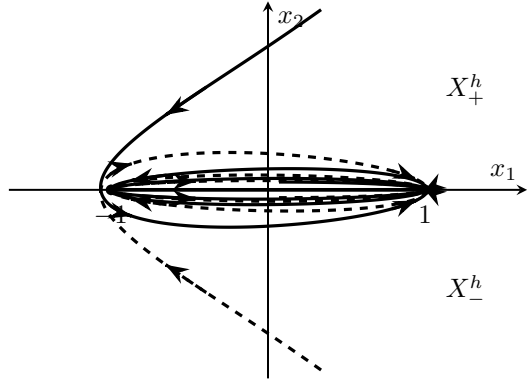

From \eqref{eq:def_Xpm} and \eqref{eq:factorisation} we get, on the chosen
sheet $\{\xi>0\}$,
\begin{equation}
  \label{eq:sum_Xpm}
  X^h_++X^h_-=-2\,a(\xi)\begin{pmatrix}1\\0\end{pmatrix}.
\end{equation}
Since $\sinh\xi>0$ on this sheet, the vector in
\eqref{eq:sum_Xpm} is nowhere vanishing and is always parallel to the gradient of $f\circ \Phi$. 

\begin{lemma}
  \label{lemma:critical_point}
  Let $\gamma\colon\mathbb{S}^1\to\mathbb{R}^2\setminus[-e_1,e_1]$ be a closed
  simple curve of class $\mathcal{C}^1$ which is not an ellipse with foci $-e_1$
  and $e_1$, and let
  \begin{equation}
    \label{eq:def_f}
    f(\theta)=\frac12\Bigl(\vert\gamma(\theta)-c_1\vert
                          +\vert\gamma(\theta)-c_2\vert\Bigr).
  \end{equation}
  Then:
  \begin{enumerate}
    \item[(i)] $f$ is non-constant, hence it has at least two critical points, a
    maximum and a minimum; in particular there exist $\theta_-,\theta_+$ with
    $f'(\theta_-)f'(\theta_+)<0$;
    \item[(ii)] $\theta$ is a critical point of $f$ if and only if
    \[
    \bigl\langle d\Phi\left(
    X^h_+(\Phi^{-1}\gamma(\theta))+X^h_-(\Phi^{-1}\gamma(\theta))\right),
    \dot{\gamma}(\theta)\bigr\rangle=0;\]
    \item[(iii)] this happens if and only if the orbit reaching $\gamma(\theta)$
    with velocity $-d\Phi (X^h_+)$, which is asymptotic to
    $\Gamma_h$ in the past, is reflected by $\gamma$ at $\gamma(\theta)$ into
    the orbit leaving $\gamma(\theta)$ with velocity
    $d \Phi (X^h_-)$, which is asymptotic to $\Gamma_h$ in the future.
  \end{enumerate}
  If, moreover, $\gamma$ is real-analytic, the critical points of $f$ are finitely
  many and isolated.
\end{lemma}

\begin{proof}
  (i) If $f$ were constant, say $f\equiv c$, then $\gamma$ would be contained in
  the level set $\{f=c\}$, which by \eqref{eq:function_ellipse}, since $2f=r_1+r_2$, is the ellipse with foci
  $c_1,c_2$ through any of its points; being a simple closed curve, $\gamma$
  would coincide with it, against the hypothesis. Hence $f$ is a
  non-constant $\mathcal C^1$ function on $\mathbb{S}^1$, so it attains a
  maximum and a minimum at distinct points, both critical. 

  (ii) Since $\gamma$ takes values in $\mathbb{R}^2\setminus[-e_1,e_1]$ and $\Phi$
  restricts to a diffeomorphism from $\{\xi>0\}$ onto that set, the lift
  $\tilde\gamma=\Phi^{-1}\circ\gamma$ is well defined. Set
  $\tilde f=f\circ\Phi=\cosh\xi$. Then
  \[
    f'(\theta)
      =d\tilde f_{\tilde\gamma(\theta)}\bigl(\dot{\tilde\gamma}(\theta)\bigr)
      =\bigl\langle\nabla_{\xi,\eta}\tilde f,
                     \dot{\tilde\gamma}(\theta)\bigr\rangle .
  \]
  Hence $f'(\theta)=0$ if and only if
  $\langle\nabla_{\xi,\eta}\tilde f,\dot{\tilde\gamma}(\theta)\rangle=0$,
  which by \eqref{eq:sum_Xpm} is equivalent to the condition in (ii).

  (iii) At $\tilde\gamma(\theta)$, set
  $v_{\mathrm{in}}=-X^h_+$ and $v_{\mathrm{out}}=X^h_-$.  The two vectors have the same length and
  $v_{\mathrm{out}}-v_{\mathrm{in}}=X^h_-+X^h_+$ is orthogonal to $\dot{\tilde\gamma}(\theta)$
  precisely when the condition in (ii) holds.
  This says exactly
  that, at critical points of $f$, $v_{\mathrm{out}}$ is the mirror image of $v_{\mathrm{in}}$ across the
  tangent line to $\tilde\gamma$, that is, that the two orbits are exchanged by
  the elastic reflection. 
  
  Note that equality of lengths and orthogonality to the
  tangent are conformally invariant conditions at a fixed point. Since $\Phi$ is
  conformal by \eqref{eq:euclidean_elliptic}, they may therefore be checked either
  in the $(\xi,\eta)$ variables with the Euclidean metric, as above, or in the
  original coordinates.

  The last assertion is the standard fact that a non-constant real-analytic
  function on $\mathbb{S}^1$ has finitely many zeros of its derivative.
\end{proof}

\begin{remark}
  \label{rem:first_order_contact}
  Geometrically, the critical points of $f$ are precisely the points at
  which $\gamma$ has first-order contact with a member of the confocal family
  $E_c=\{x:\ f(x)=c\}$ (cf. Figure \ref{fig:inner_ellipse}). If $\gamma(\mathbb S^1)$ is itself a confocal ellipse,
  every point is a contact point of infinite order. If instead $\gamma$ is
  real-analytic and its image is not a confocal ellipse, every contact point has
  finite order and the contact points are isolated.
\end{remark}

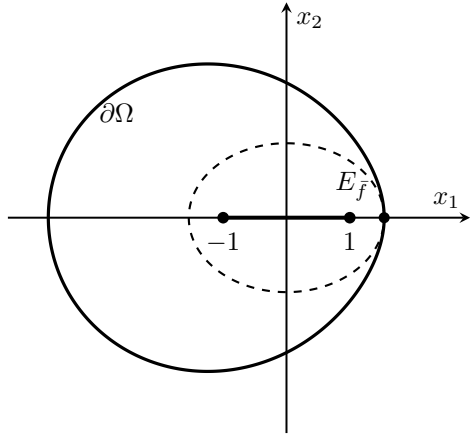
\begin{figure}\label{fig:inner_ellipse}
\centering

\begin{tikzpicture}
\begin{axis}[
    axis equal image,
    axis lines=middle,
    xlabel={$x_1$},
    ylabel={$x_2$},
    xmin=-4.4, xmax=2.9,
    ymin=-3.4, ymax=3.4,
    xtick={-1,0,1},
    ytick=\empty,
    samples=900,
    domain=0:6.28318530718,
    clip=false,
]

\def\xibar{1}

\addplot[
    very thick,
    no marks,
    variable=\t,
]
(
    {cosh(\xibar + sin(deg(\t/2))^2) * cos(deg(\t))},
    {sinh(\xibar + sin(deg(\t/2))^2) * sin(deg(\t))}
);

\addplot[
    thick,
    dashed,
    no marks,
    variable=\t,
]
(
    {cosh(\xibar) * cos(deg(\t))},
    {sinh(\xibar) * sin(deg(\t))}
);

\addplot[
    black,
    line width=1.5pt
]
coordinates {(-1,0) (1,0)};

\addplot[
    only marks,
    mark=*,
    mark size=1.8pt
]
coordinates {(-1,0) (1,0)};


\addplot[
    only marks,
    mark=*,
    mark size=1.7pt
]
coordinates {({cosh(\xibar)},0)};

\node[anchor=south east] at (axis cs:-2.25,1.35) {$\partial\Omega$};
\node[anchor=south west] at (axis cs:0.62,0.12) {$E_{\bar f}$};

\end{axis}
\end{tikzpicture}
\caption{The inner confocal ellipses touching the boundary $\partial\Omega$ at a minimum point of $f$}
\end{figure}

\section{The billiard map}
\label{sec:billiard_map}

Before constructing the symbolic dynamics, we record the basic facts about the billiard map of 2-centre billiards in bounded domains $\Omega$. In contrast with Birkhoff billiards in strictly convex domains, solutions of the 2-centre problem starting from a boundary point may never return to $\partial \Omega$ or do so tangentially.  
	
In this section we will discuss these features. In particular, we will focus on the role of the stable manifold of 
$\Gamma_h$, the hyperbolic orbit appearing in Proposition \ref{prop:collision_orbit}, and discuss some geometric conditions  to prevent trajectories from returning to the boundary tangentially, such as 
strict $g_h$-convexity. This will later provide a convenient sufficient
condition for global well-posedness of the billiard map.  The symbolic-dynamics
construction of Section~\ref{sec:approximation}, however, does not require this
convexity assumption.

Throughout, $h\ge0$ is fixed, $\Omega$ is bounded with $[-e_1,e_1]\subset \Omega\setminus \partial \Omega$
and $\partial\Omega$ is of class $\mathcal{C}^1$. We work in the elliptic-hyperbolic coordinates $(\xi,\eta)$ and identify $\mathbb{R}^2\setminus[-e_1,e_1]$ with $\{\xi>0\}$.
We denote the flow of $K$, the Hamiltonian function defined in \eqref{eq:def_K}, by $\Psi^t$ and refer to the latter as the \textit{regularised flow}.

We denote by $M_h$ the set of inward-pointing velocities based at a point of $\partial \Omega$ and having energy $h$. As for Birkhoff billiards, we define a parametrization of $M_h$ using the angle with the normal to $\partial \Omega$.

For $\theta\in\mathbb{S}^1$ let $N(\theta)$ be the inward unit normal to
$\partial\Omega$ at $\gamma(\theta)$, $\nu(\theta)=-N(\theta)$ the outward one,
and $T(\theta)=\dot\gamma(\theta)/\vert\dot\gamma(\theta)\vert$ the unit tangent.

\begin{definition}
  \label{def:theta_alpha}
  For an inward velocity $v$ at $\gamma(\theta)$ let
  $\alpha\in(-\tfrac\pi2,\tfrac\pi2)$ be defined by
  \begin{equation}
    \label{eq:def_alpha}
    \frac{v}{\vert v\vert}=\cos\alpha\ N(\theta)+\sin\alpha\ T(\theta) .
  \end{equation}
  This identifies the set of inward data on the boundary having energy $h$  with the open cylinder
  \[
    M_h\ \cong\ \mathbb{S}^1\times\Bigl(-\frac\pi2,\frac\pi2\Bigr),
    \qquad z=(\theta,\alpha),
  \]
  whose two boundary circles $\{\alpha=\pm\tfrac\pi2\}$ in $\overline{M_h}$ we
  call the \emph{grazing circles}.
\end{definition}

Two remarks. First, on the common domain of definition, the involution
conjugating $\mathcal{B}$ and $\mathcal{B}^{-1}$, which sends an inward datum to
the inward datum associated with the time-reversed billiard trajectory, reads
simply
\begin{equation}
  \label{eq:involution_alpha}
  \iota(\theta,\alpha)=(\theta,-\alpha).
\end{equation}
Second, since $\Phi$ is conformal by \eqref{eq:euclidean_elliptic}, it preserves
angles: writing $\tilde\gamma=\Phi^{-1}\circ\gamma$ for the lifted boundary and
$\tilde N,\tilde T$ for its inward normal and tangent in the Euclidean
$(\xi,\eta)$-plane, the angle \eqref{eq:def_alpha} may equivalently be computed in $\mathcal{C}$.

	We introduce now a notion of \textit{first return} to $\partial \Omega$. Since the flow $\Psi^t$ is smooth only on $\{K=0\} \subset T^*\mathcal{C}$, it is most convenient to work there.  Denote by $\pi: T^*\mathcal{C} \to \mathcal{C}$ the base projection. 
		\begin{definition}
  \label{def:exit_time}
  For $z\in \{K=0\}$ with $\pi (z) \in\Phi^{-1}(\overline\Omega)$ let
  \[
    \tau(z):=\inf\bigl\{t>0:\pi\bigl(\Psi^t(z)\bigr)\in \Phi^{-1}(\partial\Omega) \bigr\}
    \ \in(0,+\infty],
  \]
  the \emph{return time}. 
\end{definition}

In particular, for $z=z(\theta,\alpha)\in M_h$ with $\tau(z)<\infty$ such that the velocity of $\Phi(\Psi^{\tau(z)}(z))$ is transversal to $\partial \Omega$, we set $\mathcal{B}(\theta,\alpha)=(\theta',\alpha')$. Here,
$\theta'$ is determined by $\gamma(\theta')=\Phi\bigl(\Psi^{\tau(z)}(z)\bigr)$ and $\alpha'$ is the
angle \eqref{eq:def_alpha} of the reflected velocity at $\gamma(\theta')$.
This provides a general notion of \textit{domain} $\mathcal{D}\subset \mathbb{S}^1\times\Bigl(-\frac\pi2,\frac\pi2\Bigr)$ for the billiard map $\mathcal{B}$.

The previous definition entails that a point $(\theta,\alpha)$ fails to belong to the domain of $\mathcal{B}$ if the corresponding $z(\theta,\alpha )\in M_h$ either does not return to $\partial \Omega$ or does so, but tangentially. We discuss now the first case.

\subsection{Trapped trajectories and the stable directions at the boundary}

The first question we address is: \textit{which solutions never come back to the boundary?} Since we are essentially dealing with two one dimensional integrable Hamiltonian system and $\Phi^{-1}(\overline\Omega)$ is compact, this amounts to classifying all the periodic orbits inside $\Phi^{-1}(\overline\Omega)$. We already know that there are just two periodic orbits: $\Gamma_h$ and $\Gamma_h^{-1}$, so the points for which $\tau(z) =+\infty$ must belong to their stable manifolds. This is, however, not sufficient, as is clarified by the following Proposition. 
\begin{prop}
  \label{prop:exit_time_finite}
  Let $z$ be such that either $\pi(z) \in \Phi^{-1}(\overline\Omega)$, $K(z)=0$ and set $\bar K_0:=-(\mu_1+h)$. Then
  \[
    \tau(z)=+\infty\quad\Longrightarrow\quad
    z\in\Gamma_h\cup W^s(\Gamma_h)\cup W^s(\Gamma_h^{-1}).
  \]
  More precisely, writing $K_0=K_0(z)$ and
  $\xi_{\max}:=\max_{x\in\overline\Omega}\xi(x)$, one has:
  \begin{enumerate}
    \item[(i)] if $K_0>\bar K_0$, then $\tau(z)<+\infty$ and, for any lift of
    $z$ to $\{K=0\}$,
    \[
      \tau(z)\le
      \frac{2\xi_{\max}}{\sqrt{2(K_0-\bar K_0)}};
    \]
    \item[(ii)] if $K_0<\bar K_0$, then $\tau(z)<+\infty$;
    \item[(iii)] if $K_0=\bar K_0$, then $\tau(z)=+\infty$ on $\Gamma_h$ and $\Gamma_h^{-1}$ and
    every branch of separatrix asymptotic in the past has finite exit time.  A
    branch asymptotic to $\Gamma_h$ or $\Gamma_h^{-1}$ in the future has infinite return time precisely when its positive
    semi-orbit remains in $\overline\Omega$.
  \end{enumerate}
\end{prop}

\begin{proof}
  Choose a lift to the regularised cylinder.  Since $\Omega$ is bounded, every
  lift of $\overline\Omega$ is contained in $\{|\xi|\le\xi_{\max}\}$.  We use
  Lemma~\ref{lemma:monotonicity_xi} in the form
  \[
    p_\xi^2=2\bigl(K_0+W_1(\xi)\bigr),
  \]
  recalling that $W_1$ is even and has its minimum
  $W_1(0)=-\bar K_0$ at $\xi=0$.

  If $K_0>\bar K_0$, then
  $|\dot\xi|=|p_\xi|\ge\sqrt{2(K_0-\bar K_0)}$ and the sign of $p_\xi$ is
  constant.  Thus $\xi$ is strictly monotone and crosses one of the levels
  $\xi=\pm\xi_{\max}$ in the stated time; before doing so the projected orbit
  must meet $\partial\Omega$.

  If $K_0<\bar K_0$, the orbit is confined to one component of
  $\{|\xi|\ge\xi_*\}$ and has a unique turning point.  After the turning point
  $|\xi|$ tends to infinity.  The regularised time needed to move from the
  turning point to any fixed level is finite, because
  $\bigl(2(K_0+W_1(\xi))\bigr)^{-1/2}$ has the integrable square-root
  singularity at $\xi=\xi_*$.  Hence the orbit must leave the bounded domain in
  finite time.

  Finally, if $K_0=\bar K_0$, the $(\xi,p_\xi)$-level consists of the saddle and
  its four separatrix branches.  The two backward-asymptotic branches satisfy
  $|\xi(t)|\to\infty$ as $t\to+\infty$ and therefore exit in finite time.  The
  saddle projects to $\Gamma_h\subset\Omega$ and has infinite return time.  On a
  forward-asymptotic branch one has $\xi(t)\to0$, but this fact alone does
  \emph{not} imply that the projected orbit remains in a general non-convex
  domain; the last assertion is therefore exactly the stated alternative.
\end{proof}

We have seen that the tangents of the stable manifolds of $\Gamma_h$ and $\Gamma_h^{-1}$ give two vector fields $X_\pm^h$  on $\mathbb{R}^2\setminus[-e_1,e_1]$, defined in \eqref{eq:def_Xpm}. Thus, two leaves pass through every point on $\partial \Omega$: one of $W^s(\Gamma_h)$ and one of $W^s(\Gamma_h^{-1})$. However, as we already remarked,  these vectors can be tangent to the boundary or outward-pointing. Thus, we give the following definition.

\begin{definition}
	\label{def:alpha_pm}
	For each sign, set
	\[
	D_h^\pm:=\bigl\{\theta\in\mathbb S^1:
	\langle X^h_\pm(\tilde \gamma(\theta)),\tilde N(\theta)\rangle>0\bigr\}.
	\]
	If $\theta\in D_h^\pm$, let $\alpha_h^\pm(\theta)\in(-\pi/2,\pi/2)$ be the
	angle \eqref{eq:def_alpha} of $X^h_\pm(\tilde \gamma(\theta))$, and set
	\begin{equation}
		\label{eq:def_graphs}
		\mathcal G_h^\pm:=
		\bigl\{(\theta,\alpha_h^\pm(\theta)):\theta\in D_h^\pm\bigr\}
		\subset M_h,
		\qquad
		\mathcal W_h^s:=\mathcal G_h^+\cup\mathcal G_h^-.
	\end{equation}
	We call these the inward stable boundary graphs of $\Gamma_h$ and $\Gamma^{-1}_h$. We denote by $\iota(\mathcal W_h^s)$ their reflection about the normal $N(\theta)$.
	
\end{definition}

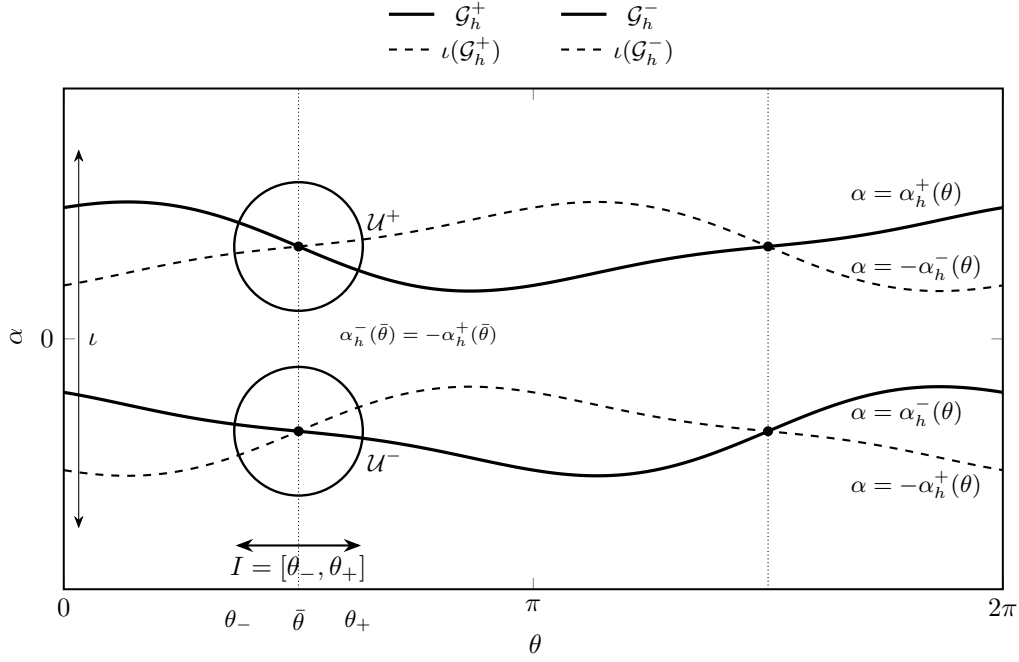
\begin{figure}
\centering
%
%
%
%
%
%
%

\begin{tikzpicture}
\begin{axis}[
    width=14cm,
    height=8.2cm,
    xmin=0, xmax=6.28318530718,
    ymin=-1.03, ymax=1.03,
    axis lines=box,
    xlabel={$\,\theta$},
    ylabel={$\,\alpha$},
    xtick={
        0,
        3.14159265359,
        6.28318530718
    },
    xticklabels={
        $0$,
        $\pi$,
        $2\pi$
    },
    ytick={0},
    samples=800,
    clip=false,
    legend style={
        draw=none,
        fill=none,
        font=\small,
        at={(0.5,1.03)},
        anchor=south,
        legend columns=2,
        /tikz/every even column/.append style={column sep=7mm}
    },
]


\addplot[
    black,
    very thick,
    domain=0:6.28318530718,
    variable=\t
]
{
  0.38
  + 0.16*cos(deg(\t))
  + 0.05*sin(deg(2*\t))
};
\addlegendentry{$\mathcal G_h^+$}

\addplot[
    black,
    very thick,
    domain=0:6.28318530718,
    variable=\t
]
{
 -0.38
 + 0.16*cos(deg(\t))
 - 0.05*sin(deg(2*\t))
};
\addlegendentry{$\mathcal G_h^-$}


\addplot[
    black,
    thick,
    dashed,
    domain=0:6.28318530718,
    variable=\t
]
{
 -0.38
 - 0.16*cos(deg(\t))
 - 0.05*sin(deg(2*\t))
};
\addlegendentry{$\iota(\mathcal G_h^+)$}

\addplot[
    black,
    thick,
    dashed,
    domain=0:6.28318530718,
    variable=\t
]
{
  0.38
 - 0.16*cos(deg(\t))
 + 0.05*sin(deg(2*\t))
};
\addlegendentry{$\iota(\mathcal G_h^-)$}


\addplot[
    black,
    thin,
    densely dotted
]
coordinates {
    (1.57079632679,-1.03)
    (1.57079632679, 1.03)
};

\addplot[
    black,
    thin,
    densely dotted
]
coordinates {
    (4.71238898038,-1.03)
    (4.71238898038, 1.03)
};

\node[anchor=north,font=\small]
at (axis cs:1.57079632679,-1.075)
{$\bar\theta$};

\node[anchor=north,font=\small]
at (axis cs:1.17079632679,-1.075)
{$\theta_-$};

\node[anchor=north,font=\small]
at (axis cs:1.97079632679,-1.075)
{$\theta_+$};


%
%

\addplot[
    only marks,
    mark=*,
    mark size=1.5pt
]
coordinates {
    (1.57079632679, 0.38)
    (1.57079632679,-0.38)
    (4.71238898038, 0.38)
    (4.71238898038,-0.38)
};

\node[
    draw,
    circle,
    line width=0.9pt,
    minimum size=17mm,
    inner sep=0pt
] at (axis cs:1.57079632679, 0.38) {};

\node[
    draw,
    circle,
    line width=0.9pt,
    minimum size=17mm,
    inner sep=0pt
] at (axis cs:1.57079632679,-0.38) {};

%

\node[anchor=west,font=\small]
at (axis cs:1.97, 0.49) {$\mathcal U^+$};

\node[anchor=west,font=\small]
at (axis cs:1.97,-0.50) {$\mathcal U^-$};



\node[anchor=west,font=\scriptsize]
at (axis cs:1.78,0.02)
{$\alpha_h^-(\bar\theta)=-\alpha_h^+(\bar\theta)$};


\node[anchor=west,font=\small]
at (axis cs:5.2,0.6)
{$\alpha=\alpha_h^+(\theta)$};

\node[anchor=west,font=\small]
at (axis cs:5.2,0.3)
{$\alpha=-\alpha_h^-(\theta)$};

\node[anchor=west,font=\small]
at (axis cs:5.2,-0.3)
{$\alpha=\alpha_h^-(\theta)$};

\node[anchor=west,font=\small]
at (axis cs:5.2,-0.6)
{$\alpha=-\alpha_h^+(\theta)$};

\draw[
    <->,
    >=Stealth,
    thin
]
(axis cs:0.1,0.78) -- node[right] {$\iota$} (axis cs:0.1,-0.78);

\draw[
    <->,
    >=Stealth,
    line width=0.8pt
]
(axis cs:1.15,-0.85)
--
node[below] {$I=[\theta_-,\theta_+]$}
(axis cs:1.99,-0.85);

\end{axis}
\end{tikzpicture}
\caption{Schematic phase portrait in the non-degenerate case. The solid curves are the stable graphs $\mathcal G_h^\pm$ , while the dashed curves are their images under the involution $\iota$. 
	The circled cross-like intersections correspond to the minimum point (in this case unique) of $f$ and will contain the symbolic dynamics we will build in Section \ref{sec:symbolic}}
\end{figure}
The next lemma shows that the definition is well posed and that the domains $D^\pm_h$ are open sets.
\begin{lemma}
  \label{lem:alpha_formula}
  Write $\tilde\gamma(\theta)=(\xi(\theta),\eta(\theta))$, $'=d/d\theta$, and let
  $a,b$ be as in \eqref{eq:def_Xpm}.  On $D_h^\pm$ one has
  \begin{equation}
    \label{eq:alpha_pm}
    \tan\alpha^{\pm}_h(\theta)
    =\frac{\langle X^h_\pm,\tilde T\rangle}
          {\langle X^h_\pm,\tilde N\rangle}
    =\frac{-a\,\xi'\pm b\,\eta'}
           {\varsigma\bigl(-a\,\eta'\mp b\,\xi'\bigr)},
    \qquad
    \tilde N=\varsigma\,\frac{(\eta',-\xi')}{|\tilde\gamma'|},
  \end{equation}
  where $\varsigma\in\{\pm1\}$ is chosen so that $\tilde N$ points inwards.
  The sets $D_h^\pm$ are open, the functions $\alpha_h^\pm$ have the same
  regularity as $\dot\gamma$, and the two graphs are disjoint wherever both are
  defined.
\end{lemma}

\begin{proof}
  The first equality is \eqref{eq:def_alpha} together with conformality of
  $\Phi$, and the second follows from $X^h_\pm=(-a,\pm b)$.  On the boundary
  $\xi>0$, hence $a>0$ by \eqref{eq:factorisation}, while $b>0$ by
  \eqref{eq:pos_peta}; consequently $(-a,b)$ and $(-a,-b)$ are never parallel.
  Openness and regularity follow from the strict inward-pointing condition.
\end{proof}

\subsection{Transit time near the collision orbit}
\label{subsec:transit}
The aim of this section is to estimate the first return time of trajectories of the two-centre problem to the boundary of $\Omega$. We stress that \textit{time} always refers to \textit{regularised time}, i.e. we use the coordinates $(\xi,\eta)$ and the Hamiltonian function $K$ defined in Section \ref{sec:two_centres}.

To do so, we first place a small ellipse $\mathcal{E}_0 = \{x: f(x)=\cosh\delta_0\}$ inside $\Omega$ and we look at the first return time there.  Up to some bounded error, the return time will be the same.
In the $(\xi,\eta)$ coordinates, $\mathcal{E}_0$ lifts to two disjoint circles corresponding to $\{\xi=\pm\delta_0\}$. Set
\begin{equation}
  \label{eq:def_rho}
  \rho:=|K_0-\bar K_0|,
  \qquad \bar K_0=-(\mu_1+h),
\end{equation}
and recall $\lambda=\sqrt{\mu_1+2h}$ from \eqref{eq:def_lambda}.

For $K_0$ sufficiently close to $\bar K_0$, let
$T_{\delta_0}(K_0)$ denote the time spent in 
 $\{|\xi|\le\delta_0\}$ before returning to the boundary: if
$K_0<\bar K_0$ this is twice the time needed to travel from $\xi=\delta_0$ to the minimum of the $\xi$ coordinate (i.e. its \textit{turning point}), while if $K_0>\bar K_0$ it is the time needed to travel from $\xi=\delta_0$ to
$\xi=-\delta_0$.

\begin{lemma}
  \label{lem:transit_time}
  For $\delta_0>0$ sufficiently small there is a constant $C_{\delta_0}$ such
  that
  \begin{equation}
    \label{eq:tau_vs_rho}
    \left|T_{\delta_0}(K_0)
      -\frac1\lambda\log\frac1\rho\right|
    \le C_{\delta_0}
  \end{equation}
  \label{eq:tau_two_sided}
  whenever $0<\rho\ll1$.  If $K_0<\bar K_0$ and $\xi_{\min}$ is the turning
  point, then, with
  $\Lambda_0:=\bigl(\max_{[0,\delta_0]}W_1''\bigr)^{1/2}$,
  \begin{equation}
    \label{eq:two_sided_W}
    \frac{\lambda^2}{2}(\xi^2-\xi_{\min}^2)
    \le W_1(\xi)-W_1(\xi_{\min})
    \le \frac{\Lambda_0^2}{2}(\xi^2-\xi_{\min}^2),
  \end{equation}
  for $\xi\in[\xi_{\min},\delta_0]$, and
  \[
    \frac{\sqrt{2\rho}}{\Lambda_0}
    \le\xi_{\min}\le
    \frac{\sqrt{2\rho}}{\lambda}.
  \]
\end{lemma}

\begin{proof}
  Put $F(\xi):=W_1(\xi)-W_1(0)$.  Since
  $F(0)=F'(0)=0$ and $F''(0)=\lambda^2$, one has
  \[
    F(\xi)=\frac{\lambda^2}{2}\xi^2+O(\xi^4)
    \qquad (\xi\to0).
  \]
  Moreover $W_1'(\xi)=\sinh\xi(\mu_1+2h\cosh\xi)$ and hence, after reducing
  $\delta_0$ if necessary,
  $\lambda^2\xi\le W_1'(\xi)\le\Lambda_0^2\xi$ on
  $[0,\delta_0]$.  Integration yields \eqref{eq:two_sided_W}; setting
  $\xi=\xi_{\min}$ in $F(\xi_{\min})=\rho$ gives the stated bounds on the
  turning point.

  Introduce the local coordinate
  \[
    y=y(\xi):=\frac{\sqrt{2F(\xi)}}{\lambda},\qquad \xi\ge0.
  \]
  Then $y=\xi+O(\xi^3)$ and $d\xi/dy=1+O(y^2)$ uniformly on a sufficiently
  small fixed interval.  If $K_0<\bar K_0$, the turning point is characterised
  by $y_{\min}=\sqrt{2\rho}/\lambda$ and
  \[
    T_{\delta_0}(K_0)=
    \frac{2}{\lambda}
    \int_{y_{\min}}^{y(\delta_0)}
      \frac{d\xi/dy}{\sqrt{y^2-y_{\min}^2}}\,dy.
  \]
  The contribution of $d\xi/dy-1=O(y^2)$ is uniformly bounded, whereas the
  principal part is
  \[
    \frac{2}{\lambda}
    \operatorname{arcosh}\!\left(
      \frac{y(\delta_0)}{y_{\min}}
    \right)
    =\frac1\lambda\log\frac1\rho+O(1).
  \]
  If $K_0>\bar K_0$, the same change of variable gives
  \[
    T_{\delta_0}(K_0)=
    \frac{2}{\lambda}
    \int_0^{y(\delta_0)}
      \frac{d\xi/dy}{\sqrt{y^2+2\rho/\lambda^2}}\,dy
    =\frac1\lambda\log\frac1\rho+O(1),
  \]
  because the principal integral is an $\operatorname{arsinh}$ and the error is
  again uniformly bounded.  This proves \eqref{eq:tau_vs_rho} in both cases.
\end{proof}

\subsection{Non-tangency}
\label{subsec:non_tangency}
In this section we address the problem of finding sufficient conditions for the global well-posedness of the billiard map, which is not needed in the construction of the symbolic dynamics in Section \ref{sec:symbolic} which will be a local construction. 

To this end, we will employ the Jacobi--Maupertuis principle: trajectories with fixed energy $h$, up to
a reparametrization of time, are geodesics of the corresponding Jacobi--Maupertuis metric $g_h$ given by
\begin{equation}
	\label{eq:def_gh}
	g_h=2\,(h+U)\,\bigl(dx_1^2+dx_2^2\bigr).
\end{equation}
To formulate the first criterion, we need first to establish some properties of $g_h$ and, when $c_i= (-1)^ie_1$, of the metric spaces $(\mathcal{C},\Phi^*g_h)$. Note that the two problems are closely related. If we combine a Euclidean isometry and a rescaling, which places the centres in $\pm e_1$ and changes the energy $h$ to a new energy $\tilde h = \tilde{h}(c_1,c_2)$, we obtain a Riemannian submersion between $(\mathbb{R}^2\setminus \{c_1,c_2\})$ with the metric $g_h$ and $\mathcal{C}\setminus \{(0,\pi),(0,0)\}$ with the metric $\Phi^*g_{\tilde{h}}$. 

We write $\mathcal{K}_g$
for the Gaussian curvature of a metric $g$. We recall the following (classical) result which asserts that all these metrics have negative curvature.
\begin{lemma}
	\label{lem:neg_curvature}
	For every $h\ge0$ the metric $g_h$ has strictly negative Gaussian curvature on
	$\mathbb{R}^2\setminus\{c_1,c_2\}$. Explicitly,
	\begin{equation}
		\label{eq:curvature}
		\mathcal{K}_{g_h}=-\frac{1}{4(h+U)^3}
		\left[\,h\Bigl(\frac{m_1}{r_1^3}+\frac{m_2}{r_2^3}\Bigr)
		+\frac{m_1m_2\,\ell^2}{r_1^3r_2^3}\,\right]\;<\;0 .
	\end{equation}
	Moreover, for any $h\ge0$, $(\mathcal{C},\Phi^*g_h)$ is a complete and negatively curved Riemannian surface.
\end{lemma}

\begin{proof}
	For a conformal metric $\Lambda\,(dx_1^2+dx_2^2)$ one has
	$\mathcal{K}=-\frac{1}{2\Lambda}\Delta\log\Lambda$, hence with
	$\Lambda=2(h+U)$
	\[
	\mathcal{K}_{g_h}=-\frac{1}{4(h+U)^3}
	\Bigl((h+U)\,\Delta U-\vert\nabla U\vert^2\Bigr).
	\]
	In the plane $\Delta(1/r)=1/r^3$, so
	$\Delta U=\frac{m_1}{r_1^3}+\frac{m_2}{r_2^3}$. Setting $A=\frac{m_1}{r_1}$,
	$B=\frac{m_2}{r_2}$ and using $\nabla A=-\frac{A}{r_1}\hat r_1$,
	$\nabla B=-\frac{B}{r_2}\hat r_2$, a direct computation gives
	\begin{equation}
		\label{eq:UDU}
		(h+U)\,\Delta U-\vert\nabla U\vert^2
		= h\Bigl(\frac{m_1}{r_1^3}+\frac{m_2}{r_2^3}\Bigr)
		+\frac{m_1m_2}{r_1r_2}
		\Bigl\vert\frac{\hat r_1}{r_1}-\frac{\hat r_2}{r_2}\Bigr\vert^2 .
	\end{equation}
	The second summand simplifies: since $(x-c_1)-(x-c_2)=c_2-c_1$ is a
	\emph{constant} vector,
	\[
	\Bigl\vert\frac{\hat r_1}{r_1}-\frac{\hat r_2}{r_2}\Bigr\vert^2
	=\frac{r_2^2+r_1^2-2\,(x-c_1)\cdot(x-c_2)}{r_1^2r_2^2}
	=\frac{\vert c_2-c_1\vert^2}{r_1^2r_2^2}
	=\frac{\ell^2}{r_1^2r_2^2},
	\]
	so that it equals $m_1m_2\ell^2/(r_1^3r_2^3)$. Both summands are nonnegative
	and the second is strictly positive, whence $\mathcal{K}_{g_h}<0$ for every
	$h\ge0$. 
	
	For the second part, assuming $c_i = (-1)^i e_1$, let us observe that thanks to \eqref{eq:euclidean_elliptic}, \eqref{eq:conformal_factor} and \eqref{eq:lambdaU}  one has that 
	\[
	\Phi^*g_h = 2(h+U\circ \Phi)(\cosh^2\xi-\cos^2\eta)(d\xi^2+d\eta^2) = 2(h r_1r_2 +\mu_1 \cosh\xi-\mu_2 \cos \eta)(d\xi^2+d\eta^2)
	\]
	which is smooth and satisfies $\Phi^*g_h\ge c(d\xi^2+d\eta^2)$ where $c>0$. Thus, it is always complete. A direct computation shows that the limit of $\mathcal{K}_{g_h}$ at $\pm e_1$ exists and is negative, concluding the proof.
\end{proof}

We turn now to the following definition of geodesic convexity.
\begin{definition}
  \label{def:gh_convex}
  $\Omega$ is \emph{strictly $g_h$-convex} if for every $q\in\partial\Omega$
  there are a neighbourhood $V$ of $q$ and a $g_h$-geodesic $\sigma_q$ through
  $q$ tangent to $\partial\Omega$ with
  $\sigma_q\cap\overline\Omega\cap V=\{q\}$.
\end{definition}

The definition involves no second derivatives, so it makes sense for a
$\mathcal{C}^1$ boundary.

\begin{prop}
  \label{prop:no_grazing}
  If $\Omega$ is strictly $g_h$-convex, then:
  \begin{enumerate}
    \item[(i)] no $g_h$-geodesic segment contained in $\overline\Omega$ can be
    tangent to $\partial\Omega$ at one of its endpoints unless it is trivial;
    \item[(ii)] $D_h^+=D_h^-=\mathbb S^1$, so
    $\mathcal G_h^\pm\subset M_h$, and every forward-asymptotic stable branch
    issued from $\partial\Omega$ remains in $\overline\Omega$;
    \item[(iii)] the billiard map is well defined and continuous on
    $M_h\setminus\mathcal W_h^s$, while $\mathcal B^{-1}$ is well defined and
    continuous on $M_h\setminus\mathcal W_h^u$;
    \item[(iv)] for every $q_1,q_2\in\overline\Omega\setminus[c_1,c_2]$ and every homotopy class
    of paths joining them in $\overline\Omega\setminus[c_1,c_2]$, the
    corresponding $g_h$-geodesic is contained in $\overline\Omega$.
  \end{enumerate}
\end{prop}

\begin{proof}
  Item (i) follows directly from Definition~\ref{def:gh_convex}: a non-trivial
  geodesic segment contained in $\overline\Omega$ and tangent at $q$ would meet
  $\overline\Omega\cap V$ in points other than $q$.

  To prove (ii), assume by contradiction that $D_h^+\neq \mathbb{S}^1$. By continuity, there exists a stable branch through $\gamma(\theta)$ that is tangent to $\partial \Omega$. By convexity, it must lie outside $\Omega$ in a small neighbourhood of $\theta$.  However, far enough in
  forward time it lies in a small neighbourhood of $\Gamma_h$, hence in the
  interior of $\Omega$, and thus it must cross $\partial\Omega$ somewhere.  Let $\gamma(\theta')$ be the first such intersection point and, without loss of generality, assume that it is a transverse intersection.
  
  We can pass to the universal cover of $\mathcal{C}$ so that $\Omega$ pulls back to an infinite strip, which is simply connected. The stable branch we fixed lifts to a length minimizing curve between $\tilde{\gamma}(\theta)$ and $\tilde {\gamma}(\theta')$ and lies completely outside the strip. For any point of  the lift of $\partial \Omega$ between $\tilde{\gamma}(\theta)$ and $\tilde {\gamma}(\theta')$ there exists a unique length minimizer joining it to $\tilde{\gamma}(\theta)$. Since the stable branch is tangent to $\partial \Omega$ and the derivative of the distance is proportional to  
  \(
  \langle \dot{ \gamma}(\varphi),v(\varphi)\rangle 
  \),
  where $v(\varphi)$ is the velocity of the length minimizer joining $\tilde{\gamma}(\theta)$ to $\tilde{\gamma}(\varphi)$, we can conclude that near $\tilde{\gamma}(\theta)$ the derivative is positive whereas near $\tilde{\gamma}(\theta')$ is negative. So, there must be a zero. This, however, contradicts the strict convexity since the first zero would give a minimizer which is tangent to $\partial\Omega$ and contained in $\Omega$. The reasoning for $D_h^-$ is completely analogous. 

  For (iii), Proposition~\ref{prop:exit_time_finite} now shows that the only
  initial conditions with infinite return time are precisely the ones in  $\mathcal W_h^s$.  If
  $z\notin\mathcal W_h^s$, its first return is finite.  Since the orbit segment before
  that first return is contained in $\overline\Omega$, it can never be tangent to $\partial \Omega$ thanks to (i).  The implicit function theorem then gives continuity of
  the return time and of the reflected datum.  The assertion for the inverse
  follows by definition of the involution \eqref{eq:involution_alpha}.

  Finally, for (iv), lift the endpoints according to the prescribed homotopy
  class and apply the same argument used in point (ii) to the unique geodesic joining
  the two lifts.
\end{proof}

Of course, 
when $\partial\Omega$ is of class $\mathcal{C}^2$, strict $g_h$-convexity becomes
a pointwise inequality. We write $\kappa_{g_h}$ for the geodesic curvature of a
curve with respect to the metric $g_h$.

\begin{lemma}
  \label{lem:curvature_inequality}
  Let $\partial\Omega$ be of class $\mathcal{C}^2$ and let $\kappa_e$ be its
  Euclidean curvature with respect to the inward normal $N$. Then
  \begin{equation}
    \label{eq:kappa_gh}
    \kappa_{g_h}
    =\frac{1}{\sqrt{2(h+U)}}
      \left(\kappa_e-\frac{\langle\nabla U,N\rangle}{2(h+U)}\right),
  \end{equation}
  and $\Omega$ is strictly $g_h$-convex if and only if
  \begin{equation}
    \label{eq:admissibility}
    \kappa_e(\theta)\ >\ \frac{1}{2\bigl(h+U(\gamma(\theta))\bigr)}
      \left(\frac{m_1}{r_1^2}\langle\hat r_1,\nu\rangle
           +\frac{m_2}{r_2^2}\langle\hat r_2,\nu\rangle\right)
    \qquad\text{for every }\theta .
  \end{equation}
\end{lemma}

\begin{proof}
	It is well known that for a conformal metric $g = \Lambda g_e$ the geodesic curvature can be computed as
	 \begin{equation*}
		\kappa_{g}=\frac{1}{\sqrt\Lambda}
		\left(\kappa_e-\frac{\partial_N\Lambda}{2\Lambda}\right).
	\end{equation*}
	
  Let us apply the formula with $\Lambda=2(h+U)$, so that
  $\partial_N\Lambda/(2\Lambda)=\langle\nabla U,N\rangle/(2(h+U))$. Since
  $\nabla U=-\tfrac{m_1}{r_1^2}\hat r_1-\tfrac{m_2}{r_2^2}\hat r_2$ and
  $N=-\nu$, one gets
  $\langle\nabla U,N\rangle=\tfrac{m_1}{r_1^2}\langle\hat r_1,\nu\rangle
  +\tfrac{m_2}{r_2^2}\langle\hat r_2,\nu\rangle$, whence
  \eqref{eq:admissibility} is $\kappa_{g_h}>0$; this is equivalent to
  Definition~\ref{def:gh_convex} by the usual second-order comparison between a
  curve of positive geodesic curvature and the geodesic tangent to it.
\end{proof}

Let us observe that at any boundary point for which
$\langle\hat r_i,\nu\rangle>0$ for both centres, the right-hand side of
\eqref{eq:admissibility} is positive, so strict $g_h$-convexity is stronger
than ordinary Euclidean convexity there.  Thus the attractive force may bend a
trajectory enough for a strictly convex boundary to be grazed.

Of course, ellipses are strictly $g_h$-convex. This can be easily seen from the explicit integration of the system in Section \ref{sec:two_centres}: confocal ellipses are given by $\{\xi=\xi_0\}$ and the $\xi$ component of solutions can have only minima. We can also compute the curvature directly using Lemma \ref{lem:curvature_inequality}.

\begin{remark}[Confocal ellipses]
  \label{rem:ellipse_admissible}
  If $\partial\Omega=\{\xi=\xi_0\}$ then, by \eqref{eq:speed_on_shell},
  $\Phi^*g_h=2\bigl(W_1(\xi)-W_2(\eta)\bigr)(d\xi^2+d\eta^2)$, and \eqref{eq:kappa_gh} in the $(\xi,\eta)$-plane gives
  \[
    \kappa_{g_h}=\frac{1}{\sqrt{2(W_1-W_2)}}\cdot\frac{W_1'(\xi)}{2(W_1-W_2)},
    \qquad
    W_1'(\xi)=\sinh\xi\,(\mu_1+2h\cosh\xi)>0
    \quad\text{for }\xi>0 .
  \]
  Hence every confocal ellipse is strictly $g_h$-convex at every energy $h\ge0$;
  and since \eqref{eq:admissibility} is open in the $\mathcal{C}^2$ topology, so
  is every domain sufficiently close to one.
\end{remark}

The last observation we make is that $g_h$-convexity for strictly convex Euclidean domains is recovered in the high-energy regime. 
\begin{prop}[Strict convexity suffices at high energy]
  \label{prop:high_energy}
  Let $\Omega$ be bounded and strictly convex with $\mathcal{C}^2$ boundary and
  $[c_1,c_2]\subset\Omega$, and set
  $\kappa_{\min}=\min_\theta\kappa_e(\theta)>0$ and
  $d=\operatorname{dist}(\partial\Omega,\{c_1,c_2\})>0$. Then $\Omega$ is
  strictly $g_h$-convex for every $h>\bar h:=\mu_1/(2d^{2}\kappa_{\min})$.
\end{prop}

\begin{proof}
  On $\partial\Omega$ one has $r_i\ge d$ and $\langle\hat r_i,\nu\rangle\le1$, so
  the right-hand side of \eqref{eq:admissibility} is at most
  $\mu_1/\bigl(2d^2(h+U)\bigr)\le\mu_1/(2d^2h)$.
\end{proof}

\subsection{Invariant curves near the boundary}
\label{subsec:lazutkin}

Theorem~\ref{thm:local_nonintegrability} says that the dynamics admits no analytic
first integral; the obstruction to integrability is the presence of a chaotic subset, built in Section \ref{sec:symbolic}. However, this subset is localized in phase space in a neighbourhood of the intersection of the stable graphs and their reflections and gives little information about the global behaviour of the billiard map $\mathcal{B}$. 
For instance, the invariant set $\Lambda_h$ of
Corollary~\ref{cor:semiconjugacy} carries positive entropy, yet nothing so far
prevents its orbits from visiting the whole of $M_h$, nor rules out that
$\mathcal{B}$ be ergodic.
In this subsection we show that the latter is not the case. Near each
of the two grazing circles the phase cylinder carries a Cantor family of
essential (i.e. non-contractible) invariant curves on which the dynamics is conjugated to a rotation. These curves are also called \textit{rotational invariant curve}. Since an essential curve separates the cylinder,
each of them bounds an invariant region which no orbit can leave. The phase
space thus splits into a regular part at its two ends and a part containing the
chaotic set in the middle. To achieve this result, we need additional assumptions: strict
$g_h$-convexity and $\mathcal{C}^7$ regularity of the boundary.

Throughout this subsection $\partial\Omega$ is parametrised by $g_h$-arclength,
it is of class $\mathcal{C}^7$, and $\Omega$ is strictly $g_h$-convex. We write
\begin{equation}
  \label{eq:grazing_angle}
  \varphi:=\frac\pi2-\alpha\ \in\ (0,\pi)
\end{equation}
for the angle between the velocity and the tangent, so that the grazing circles
of Definition~\ref{def:theta_alpha} are $\{\varphi=0\}$ and $\{\varphi=\pi\}$,
and we fix a collar
\begin{equation}
  \label{eq:def_collar}
  \mathcal{N}:=\bigl\{x\in\overline\Omega:\
  \operatorname{dist}_{g_h}(x,\partial\Omega)<\varepsilon_1\bigr\},
  \qquad
  \varepsilon_1<\operatorname{dist}_{g_h}\bigl(\partial\Omega,[c_1,c_2]\bigr),
\end{equation}
on which $g_h$ is a real-analytic Riemannian metric, the two centres lying at
positive distance from it.

The existence of the essential curves follows from the same arguments given in \cite{DiasCarneiroEtAl2024}. The only point that requires some care is to show that near-grazing orbits never leave the
collar. Once this is known, everything happens on $\mathcal{N}$ and the classical
theory applies unchanged.

\begin{lemma}[Near-grazing orbits stay in the collar]
  \label{lem:grazing_collar}
  There are $\varphi_0>0$ and $C>0$ such that every geodesic arc leaving
  $\gamma(\theta_1)$ into $\Omega$ with angle
  $0<\varphi_1<\varphi_0$ meets $\partial\Omega$ again at a point
  $\gamma(\theta_2)$ before leaving $\mathcal N$.  Moreover its $g_h$-length is
  at most $C\varphi_1$, its maximal $g_h$-distance from the boundary is at most
  $C\varphi_1^2$, and
  $|\theta_2-\theta_1|\le C\varphi_1$.  In particular the arc has winding
  number $k=0$ around $[c_1,c_2]$.  After possibly decreasing $\varphi_0$, the
  strips $\{0<\varphi<\varphi_0\}$ and
  $\{\pi-\varphi_0<\varphi<\pi\}$ are disjoint from the stable graphs.
\end{lemma}

\begin{proof}
  Work in coordinates $(s,r)$ in a fixed collar of the boundary,
  where $s$ is boundary arclength and $r\ge0$ is the inward normal distance.
  Define
  $\kappa_0:=\min_{\partial\Omega}\kappa_{g_h}>0$.  If
  $q(t)=(s(t),r(t))$ is a unit-speed geodesic issued from the boundary with
   angle $\varphi_1$, then
  \[
    r(0)=0,\qquad \dot r(0)=\sin\varphi_1,
    \qquad
    \ddot r(0)=-\kappa_{g_h}(\theta_1)\cos^2\varphi_1.
  \]
  The coefficients of the metric and their derivatives are uniformly bounded
  on the collar.  Hence, after shrinking the collar to a width $\ve_1$, one can assume
  $\ddot r(t)\le-\kappa_0/2$ as long as the geodesic remains in that collar.  
  Therefore,
  $r$ is strictly concave there.
  Let $\tau$ be the first exit time from the collar. Then, for some $a_0>0$ and for all $t\in[0,\tau]$, we have
  \[
  \ddot{r}(t)\le -\kappa_0/2, \quad 0<\dot{r}(0)\le a_0, \quad r(0)=0.
  \]
  By the standard comparison principle for ODEs, it holds that $r(t)\le a_0 t-t^2\kappa_0/4$ on $[0,\tau]$. However, the upper bound has a maximum at the point $2 a_0/\kappa_0$ and thus $r(t)\le a_0^2/\kappa_0$. Thus, since the estimates are uniform in the shooting angle $\varphi_1$, no trajectory that starts with a small enough angle can exit the collar neighbourhood. Choosing
  $\varphi_0$ so that $(\sin \varphi_0)^2<\kappa_0\varepsilon_1$ 
  proves that the whole arc lies in $\mathcal N$.
  
  Since $\sin \varphi_1\sim \varphi_1$, the argument also shows that there exist $C>0$  such that, first, $r(t)$ becomes positive and then returns to zero
  at some time
  \[
    0<t_2\le C\varphi_1,
  \]
  while $0\le r(t)\le C\varphi_1^2$ for $0\le t\le t_2$.  
  Since $|\dot s|$ is uniformly
  bounded in these coordinates, we have also
  $|\theta_2-\theta_1|\le C\varphi_1$.

  An arc contained in the collar is homotopic in
  $\overline\Omega\setminus[c_1,c_2]$ to the short boundary arc between its
  endpoints, hence has winding number $0$.  Finally set
  \[
    \xi_\partial:=\min_{\partial\Omega}\xi>0.
  \]
  By \eqref{eq:factorisation}, $a(\xi)\ge a(\xi_\partial)>0$ on the boundary,
  whereas $b$ is bounded.  Together with \eqref{eq:alpha_pm} and
  Proposition~\ref{prop:no_grazing}(ii), this yields a uniform
  $\alpha_1<\pi/2$ with $|\alpha_h^\pm|\le\alpha_1$.  Taking
  $\varphi_0<\pi/2-\alpha_1$ excludes both grazing strips from the stable
  graphs; the corresponding geometric estimate near $\varphi=\pi$ follows by
  time reversal.
  
\end{proof}

\begin{prop}[Lazutkin curves]
  \label{prop:lazutkin}
  Let $h\ge0$ and let $\Omega$ be a bounded domain with $[c_1,c_2]\subset\Omega$,
  with $\partial\Omega$ of class $\mathcal{C}^7$, and assume that $\Omega$ is strictly
  $g_h$-convex. Then there is a Cantor set
  $\mathcal{K}\subset(0,\tfrac12)\setminus\mathbb{Q}$ accumulating at $0$ such
  that for every $\varrho\in\mathcal{K}$ the billiard map $\mathcal{B}$ has a
  rotational invariant curve $C_\varrho\subset M_h$ with rotation number
  $\varrho$, on which $\mathcal{B}$ is conjugate to the rotation by $\varrho$.
  The $C_\varrho$ are graphs over $\mathbb{S}^1$ and accumulate on the grazing
  circle $\{\varphi=0\}$; the involution $\iota$ of
  \eqref{eq:involution_alpha} gives the corresponding family at
  $\{\varphi=\pi\}$.
\end{prop}

\begin{proof}
  By Lemma~\ref{lem:grazing_collar} the restriction of $\mathcal{B}$ to
  $\{0\le\varphi<\varphi_0\}$ is a well-defined map of a half-open annulus, all of
  whose orbits are contained in the collar $\mathcal{N}$, on which $g_h$ is a
  real-analytic Riemannian metric and $\partial\Omega$ is a $\mathcal{C}^7$ curve
  of strictly positive $g_h$-geodesic curvature. The hypotheses of
  \cite{DiasCarneiroEtAl2024} are therefore met on $\mathcal{N}$, and we may
  apply their results verbatim.

  By \cite[Theorem~1 and Lemma~3.2]{DiasCarneiroEtAl2024}, the local
  billiard map in the collar is a conservative twist map.  For nearby boundary
  points it is generated by the $g_h$-length $S(\theta_1,\theta_2)$ of the
  unique short geodesic chord, with
  $\partial_1S=-\cos\varphi_1$, $\partial_2S=\cos\varphi_2$ and
  $\partial_{12}S>0$; it preserves
  $d\theta\wedge\sin\varphi\,d\varphi$. By
  \cite[Proposition~3.4]{DiasCarneiroEtAl2024},
  \begin{equation}
    \label{eq:lazutkin_expansion}
    \theta_2=\theta_1+\frac{2\varphi_1}{\kappa_{g_h}(\theta_1)}+o(\varphi_1),
    \qquad
    \varphi_2=\varphi_1+o(\varphi_1),
  \end{equation}
  so the twist is uniform up to the grazing circle. With
  \eqref{eq:lazutkin_expansion} and $\mathcal{C}^7$ regularity, Douady's
  Corollary~II-2 \cite{Douady1982} yields the Cantor family of invariant curves,
  as in \cite[Theorem~2]{DiasCarneiroEtAl2024}; that they are graphs is
  Birkhoff's invariant curve theorem \cite{Birkhoff1932}.
\end{proof}

\begin{remark}
  \label{rem:lazutkin_examples}
  By Remark~\ref{rem:ellipse_admissible} every confocal ellipse is strictly
  $g_h$-convex at every energy, and \eqref{eq:admissibility} is an open condition;
  by Proposition~\ref{prop:high_energy} every strictly convex domain is strictly
  $g_h$-convex for $h>\bar h$. Proposition~\ref{prop:lazutkin} therefore applies
  to every $\mathcal{C}^7$ domain sufficiently close to a confocal ellipse, at
  every energy, and to every strictly convex $\mathcal{C}^7$ domain at high
  energy. 
  \end{remark}

\section{Approximation of the invariant manifolds}
\label{sec:approximation}

In this section we prove that the stable and unstable manifolds of the
collision--reflection orbit $\Gamma_h$ from
Proposition~\ref{prop:collision_orbit} can be shadowed by trajectories that
approach $[c_1,c_2]$ and then return to $\partial\Omega$. These approximating
arcs are indexed by their winding number around the segment, and the approximation
error decays exponentially with this number. We also prove that, after choosing
the boundary interval appropriately, these arcs are genuine billiard arcs. This
last step is what allows us to avoid any convexity assumption on $\Omega$.

Throughout this section, $h\ge 0$ is fixed, and we work in the following
setting
\begin{center}
\begin{minipage}{0.92\textwidth}{\it 
	\centering
  $\Omega\subset\mathbb{R}^2$ is a bounded domain with $[c_1,c_2]\subset\Omega$ whose boundary $\partial\Omega$ is a simple
  closed curve of class $\mathcal{C}^1$ which 
   is \emph{not} an ellipse with foci $c_1$ and $c_2$.}
\end{minipage}
\end{center}

We use the same symbol
$f$ for the function
$\frac12(|x-c_1|+|x-c_2|)$, defined on the plane, and for its pullback to $\mathbb{S}^1$ via $\gamma$, given by 
$f(\theta)=f(\gamma(\theta))$. We write
\begin{equation}
  \label{eq:def_fbar}
  \bar f:=\min_{\mathbb{S}^1}f>1,
  \qquad
  \bar\xi:=\operatorname{arccosh}\bar f>0,
  \qquad
  E_{<c}:=\{x:\ f(x)<c\},
\end{equation}
the strict inequality on the minimum holding because $\partial\Omega$ is compact and disjoint
from the segment $[c_1,c_2]$.

The first step of the construction consists in defining a suitable family of trajectories and understanding their properties. We will take again a geometric perspective and apply the Jacobi--Maupertuis principle: trajectories with fixed energy $h$, up to
a reparametrization of time, are geodesics of the corresponding Jacobi--Maupertuis metric $g_h$.
\begin{definition}
  \label{def:sk}
  Fix $\theta_0\in\mathbb{S}^1$ and $k\in\mathbb{Z}$. For
  $\theta_1,\theta_2\in\mathbb{S}^1\setminus\{\theta_0\}$ we denote by
  $\sk$ the minimiser, whose existence and uniqueness are proved below, of the
  Jacobi--Maupertuis length $\mathcal{L}_h$ among the curves in
  $\mathbb{R}^2\setminus[c_1,c_2]$ joining $\gamma(\theta_1)$ to
  $\gamma(\theta_2)$ and satisfying
  \[
    \sk\,\#\,\gamma\vert_{[\theta_2,\theta_1]}\ \sim\ \gamma^k ,
  \]
  where $\#$ denotes concatenation, $\gamma\vert_{[\theta_2,\theta_1]}$ is the
  boundary arc from $\theta_2$ to $\theta_1$ not containing $\theta_0$, and
  $\gamma^k$ is the $k$-th iterate of a loop generating
  $\pi_1\bigl(\mathbb{R}^2\setminus[c_1,c_2]\bigr)\cong\mathbb{Z}$.
\end{definition}

Three angles enter in the construction, with three different roles. An arc joining two boundary points
is not a loop, so it has no homotopy class of its own. To assign one, we close it
up along $\partial\Omega$. There are two ways of doing so, yielding closed loops whose winding number differs by one. Fixing $\theta_0$ selects one of them, namely the arc of $\partial\Omega$
running from the arrival point back to the departure point \emph{without crossing
	$\gamma(\theta_0)$}.

The angles $\theta_1$ and $\theta_2$ specify, respectively, the departure and arrival points on $\partial\Omega$. Finally, the integer $k$ counts the windings around
$[c_1,c_2]$ and its sign records the direction.

Note that the minimization is performed among curves in $\mathbb{R}^2\setminus[c_1,c_2]$, and
\emph{not} in $\overline\Omega\setminus[c_1,c_2]$. The fact that
$\sk$ nonetheless lies in $\overline\Omega$ is the content of
Proposition~\ref{prop:confinement} below. Existence and uniqueness are proved in Proposition~\ref{prop:uniqueness}. For the variational properties of minimisers in this setting, see
\cite{Bolotin1984,BolotinNegrini2001,SoaveTerracini2012,%
	BarutelloCanneoriTerracini2021,BaranziniCanneori2024}.

As a final remark, note that the curves $\sk$
depend on $h$, but since we work on a fixed energy shell we do not record this in
the notation. It remains now to check that Definition \ref{def:sk} is well posed.

\begin{prop}
  \label{prop:uniqueness}
  For all $\theta_1,\theta_2\in\mathbb{S}^1\setminus\{\theta_0\}$ and every
  $k\in\mathbb{Z}$, except for the trivial case
  $\theta_1=\theta_2$ and $k=0$, there exists exactly one non-constant
  energy-$h$ trajectory in $\mathbb{R}^2\setminus[c_1,c_2]$ joining
  $\gamma(\theta_1)$ to $\gamma(\theta_2)$ in the prescribed class, and it is
  the minimiser $\sk$.  In the trivial case, the unique length
  minimiser is the constant curve.
\end{prop}

\begin{proof}
  By Lemma~\ref{lem:neg_curvature} the pull-back
  $\Phi^*g_h$ is a real-analytic complete metric of strictly negative curvature
  on the cylinder $\mathcal C$, and $\pi_1(\mathcal C)\cong\mathbb Z$ is
  generated by a loop $\{\xi=\xi_0\}$.  Its universal cover is therefore a
  Cartan--Hadamard surface.  After choosing lifts of the two endpoints according
  to the prescribed homotopy class, there is a unique geodesic joining them,
  and this geodesic is the unique length minimiser, up to reparametrization.  If the two lifted endpoints
  coincide, which is precisely the trivial case $\theta_1=\theta_2$, $k=0$,
  this geodesic is constant.  In every other case it is non-constant and, after
  Maupertuis reparametrisation, is an energy-$h$ trajectory.

  It remains to check that every non-constant connecting geodesic stays in the
  half-cylinder $\{\xi>0\}$, which $\Phi$ maps diffeomorphically onto
  $\mathbb R^2\setminus[c_1,c_2]$.  By
  Lemma~\ref{lemma:monotonicity_xi}, if $K_0<\bar K_0$ the orbit is confined to
  one component of $\{|\xi|\ge\xi_*\}$, while if $K_0>\bar K_0$ the coordinate
  $\xi$ is strictly monotone; in the latter case, between two endpoints with
  positive $\xi$ it remains between their two positive values.  The level
  $K_0=\bar K_0$ consists of the periodic orbit and its separatrices, and a
  separatrix reaches $\xi=0$ only asymptotically, so it cannot escape $\{\xi>0\}$.  Hence the connecting geodesic remains in
  $\{\xi>0\}$, as claimed.
\end{proof}

\subsection{Shadowing of $W^s(\Gamma_h)$ and $W^s(\Gamma_h^{-1})$}
Without loss of generality, for the rest of this section we will assume that $c_1=-e_1$ and $c_2=e_1$.
Now, we discuss the behaviour of the curves $\sk$ as $\vert k \vert \to \infty$.
\begin{prop}
  \label{prop:approximation_stable_manifold}
  For any $\ve>0$, $T>0$ and $\theta_0\in\mathbb S^1$ there exists $k_0>0$
  such that, for every $|k|\ge k_0$,
  \begin{align*}
    \sup_{\theta_1,\theta_2\in\mathbb S^1\setminus\{\theta_0\}}
      \bigl\|\sk(t)-s_\infty^{\theta_1}(t)\bigr\|_{C^1[0,T]}
      &\le\ve,\\
    \sup_{\theta_1,\theta_2\in\mathbb S^1\setminus\{\theta_0\}}
      \bigl\|\skinv(t)-s_{-\infty}^{\theta_2}(t)\bigr\|_{C^1[0,T]}
      &\le\ve.
  \end{align*}
  Here $s^{\theta_1}_{\infty}$ is the branch of $W^{s}(\Gamma_h)$ issuing
  from $\gamma(\theta_1)$
  with initial regularised velocity $X^h_{\operatorname{sgn}k}$, and
  $s^{\theta_2}_{-\infty}$ is the corresponding branch of
  $W^{s}(\Gamma_h^{-1})$ at $\gamma(\theta_2)$.
\end{prop}

We write, as in Section~\ref{sec:billiard_map},
\[
  \rho(\sk):=\bigl\vert K_0(\sk)-\bar K_0\bigr\vert\ \ge0,
  \qquad \bar K_0=-(\mu_1+h),
  \qquad \lambda=\sqrt{\mu_1+2h},
\]
and we note that, $\partial\Omega$ being compact and disjoint from
$[c_1,c_2]$, the endpoints obey uniform bounds
$0<\xi_b\le\xi(\theta_j)\le\xi_B$. To prove the proposition, we shall make use of the following lemma.

\begin{lemma}
  \label{lem:rho_k}
  There exist $C,c>0$ and $k_*>0$, depending only on $h,m_1,m_2$ and
  $\Omega$, such that, for $|k|\ge k_*$ and all
  $\theta_1,\theta_2\in\mathbb S^1\setminus\{\theta_0\}$,
  \begin{equation}
    \label{eq:rho_k_exp}
    K_0(\sk)<\bar K_0,
    \qquad
    \rho(\sk)\le C e^{-c|k|}.
  \end{equation}
  In particular $\sk$ has a unique interior turning point in $\xi$.  If
  $t_k^*$ denotes the time at which that turning point is reached when the arc
  is parametrised from $\gamma(\theta_1)$, then
  $t_k^*\to+\infty$ as $|k|\to\infty$, uniformly in the endpoints.
\end{lemma}

\begin{proof}
  Since the boundary is compact and disjoint from $[-e_1,e_1]$, its lift
  satisfies $0<\xi_b\le\xi(\theta)\le\xi_B$.  The separation constant of a
  trajectory meeting the boundary also ranges in a fixed compact interval,
  because on $\{K=0\}$ both momenta are bounded at such points by
  \eqref{eq:speed_on_shell}.

  We first show that  $K_0<\bar K_0$ and
  $K_0\to\bar K_0$.  If $K_0\ge\bar K_0$, then $\xi$ is monotone by
  Lemma~\ref{lemma:monotonicity_xi}; since by construction it cannot cross the segment $[-e_1,e_1]$, it remains
  in $[\xi_b,\xi_B]$, and there $|p_\xi|$ is bounded away from zero.  The travel
  time and hence the variation of $\eta$ are therefore uniformly bounded, so this can occur for bounded $k$ only.  
  
  If instead
  $K_0\le\bar K_0-\rho_0$ for some fixed $\rho_0>0$, the turning point
  satisfies $\xi_{\min}\ge c(\rho_0)>0$.  As done in Lemma~\ref{lem:transit_time}, one can show that this gives a uniform control on the time spent in the strip $[\xi_{\min},\xi_B]$ by $\sk$. In turn, this gives a bound on the variation of $\eta$ and on $k$.
  Consequently, $|k|$ large enough implies
  $K_0<\bar K_0$ and that $\bar K_0-K_0$ is arbitrarily small.

We deal now with the exponential bound. 
  Fix $0<\delta_0<\xi_b$.  We have just mentioned that for $\rho$ small (i.e. $k$ large) the curve $\sk$ enters the annulus $\{0<\xi< \delta_0\}$. The time that $\sk$ takes to travel from $\partial \Omega$ to $\{\xi=\delta_0\}$ is bounded by some uniform constant. However, thanks to Lemma~\ref{lem:transit_time}, the time spent inside
  $\{\xi\le\delta_0\}$ grows like
  \[
    T_{\delta_0}(K_0)=\frac1\lambda\log\frac1\rho+O(1).
  \]
  Moreover, for $K_0$ close to $\bar K_0$,
  $p_\eta^2/2=-K_0-W_2(\eta)$ is bounded above and below by positive constants,
  uniformly in $\eta$.  Thus the total variation of $\eta$ is comparable with
  the time spent inside $\{\xi\le \delta_0\}$.  By the definition of $\sk$, that variation
  equals $2\pi|k|+c$, where $c\in [-\pi,\pi]$.
  Hence
  \[
    |k|\le C_1\log\frac1\rho+C_2,
  \]
  which is equivalent, after changing the constants, to
  \eqref{eq:rho_k_exp}.

  Finally, the time needed to go from $\xi=\delta_0$ to the turning point is one half of the time spent in $\{\xi\le \delta_0\}$, up to a uniformly bounded term.
  Therefore $t_k^*\to+\infty$ uniformly as $|k|\to\infty$.
\end{proof}

\begin{proof}[Proof of Proposition~\ref{prop:approximation_stable_manifold}]
  We work on $\{K=0\}$, in regularised time.  For $|k|$ large,
  Lemma~\ref{lem:rho_k} gives $K_0(\sk)<\bar K_0$ and we can assume its $\xi$ component to be decreasing.  Both $\sk$ and
  $s^{\theta_1}_\infty$ start at the same point and evolve under the same
  real-analytic vector field: the Hamiltonian vector field of $K$, denoted by $X_K$. Their initial momenta, as functions of the separation constant, are
  \[
    p_\xi(0)=-\sqrt{2\bigl(K_0+W_1(\xi(\theta_1))\bigr)},
    \qquad
    p_\eta(0)=\operatorname{sgn}(k)
      \sqrt{2\bigl(-K_0-W_2(\eta(\theta_1))\bigr)}.
  \]
  At $K_0=\bar K_0$ the two radicands are uniformly bounded away from zero on
  the boundary, as observed in  \eqref{eq:factorisation} and \eqref{eq:pos_peta}.  Hence, the
  initial momentum depends Lipschitz-continuously on $K_0$, uniformly in
  $\theta_1$, and
  \[
    |z_{\sk}(0)-z_{s_\infty}(0)|\le C\rho(\sk).
  \]

  For a fixed $T$, the family of stable arcs
  $s^{\theta_1}_\infty([0,T])$, with $\theta_1\in\mathbb{S}^1$, lies in a
  compact subset of $\{\xi>0\}$.  By the uniform divergence of $t_k^*$ in
  Lemma~\ref{lem:rho_k}, the same is true for the corresponding initial pieces
  of $\sk$ when $|k|$ is large.  Uniform continuous dependence and Gronwall's
  inequality therefore give
  \[
    \|z_{\sk}-z_{s_\infty}\|_{C^0[0,T]}
      \le C(T)\rho(\sk),
  \]
  uniformly in both endpoints.  Since $\dot z=X_K(z)$, the same estimate gives
  $C^1$ convergence of the configuration components.  Using
  Lemma~\ref{lem:rho_k} yields the exponential bound
  \[
    \|\sk-s^{\theta_1}_\infty\|_{C^1[0,T]}
      \le C(T)e^{-c|k|}.
  \]
  The second estimate follows by applying the same argument to the time-reversed
  arc, with the endpoints exchanged.
\end{proof}

\subsection{The stable manifolds at minimum points}
\label{subsec:confinement}

The arcs $\sk$ were defined by an unconstrained minimisation, thus, nothing so far
prevents them from leaving $\Omega$. We now show that, if we restrict to initial conditions in a suitably chosen segment of the boundary near the minimum set of $f$, they do
not. We remark that no convexity hypothesis on $\partial\Omega$ is needed beyond the standing ones.
We begin with some simple but useful facts.

\begin{lemma}[No immediate return to the boundary]
  \label{lem:no_return}
  Let $\alpha_0<\pi/2$.
  Then there exists $t_1>0$ such that every orbit of $\Psi^t$ issued from
  $\gamma(\theta)$, $\theta\in \mathbb{S}^1$, with initial velocity making an angle at
  most $\alpha_0$ with the inward normal $N(\theta)$, satisfies
  \[
    \Phi(\Psi^t z)\in\Omega,
    \qquad 0<t\le t_1 .
  \]
\end{lemma}

\begin{proof}
  Since $\partial \Omega$ is compact and disjoint from the centres, the flow is
  smooth on a neighbourhood of the corresponding set of initial states.
  Let $\rho$ be a $C^1$ defining function for $\Omega$,
  chosen so that $\Omega=\{\rho>0\}$ and $\nabla\rho$ points inward on
  $\partial\Omega$.  The angle condition and compactness give
  \[
    \frac{d}{dt}\rho\bigl(\Phi(\Psi^t z)\bigr)\Big|_{t=0}
    =\langle \nabla\rho(\Phi(z)),v(z)\rangle \ge c>0
  \]
  uniformly over all admissible initial states.  By continuity of the flow,
  this inequality remains positive for $0\le t\le t_1$, after decreasing
  $t_1$ if necessary.  Hence
  $\rho(\Phi(\Psi^t z))>0$ for every $0<t\le t_1$.
\end{proof}

Recall that $\bar{f}$ denotes the minimum of the function $f$ on $\partial \Omega$. The following Lemma shows that, near the minimum set of $f$ on $\partial \Omega$, the fields $X^h_\pm$ are always uniformly transverse to the boundary and thus, Lemma \ref{lem:no_return} applies. We use the notation $\tilde{\gamma}$ for the curve satisfying $\Phi(\tilde \gamma) =\gamma$ and having positive $\xi$ coordinate.

\begin{lemma}
	\label{lem:cone_free}
	Let $A:=\{\theta\in\mathbb S^1:f(\theta)=\bar f\}$.  For every
	$\theta\in A$,
	\[
	\langle X^h_\pm(\tilde \gamma(\theta)),\tilde N(\theta)\rangle
	=a(\bar\xi)>0,
	\]
	and therefore
	$\tan\alpha_h^\pm(\theta)=\pm b(\eta(\theta))/a(\bar\xi)$ and
	$|\alpha_h^\pm(\theta)|<\pi/2$.  Consequently there exist a closed
	neighbourhood $J$ of $A$ and $\alpha_1<\pi/2$ such that
	$|\alpha_h^\pm(\theta)|\le\alpha_1$ for every $\theta\in J$.
\end{lemma}

\begin{proof}
	Since $f=\cosh\xi$, every $\theta\in A$ satisfies
	$\xi(\theta)=\bar\xi$ and $\xi'(\theta)=0$.  Hence $\partial\Omega$ is
	tangent at $\gamma(\theta)$ to the confocal ellipse
	$\{\xi=\bar\xi\}$.  As $E_{<\bar f}\subset\Omega$, its inward normal is
	$-\partial_\xi$.  Therefore
	\[
	\langle X_\pm^h(\tilde\gamma(\theta)),\tilde N\rangle=a(\bar\xi)>0,
	\]
	since $X_\pm^h=(-a,\pm b)$.  The formula for $\alpha_h^\pm$ follows
	immediately, and compactness of $A$ gives the uniform estimate on a
	neighbourhood of $A$.
\end{proof}

In view of the Miranda argument of Section~\ref{sec:symbolic}, we explicitly note that a neighbourhood of the  \emph{minimum} set of $f$ provides the sign change of $f'$ we will require. 

\begin{lemma}
  \label{lem:choice_of_I}
  For every $\delta>0$ there exist $\theta_-<\theta_+$ such that, setting
  $I=[\theta_-,\theta_+]$,
  \begin{equation}
    \label{eq:choice_of_I}
    f'(\theta_-)<0<f'(\theta_+),
    \qquad
    \max_I |f'|<\delta,
    \qquad
    \max_I f\le \bar f+\delta .
  \end{equation}
\end{lemma}

\begin{proof}
  Let $A=\{f=\bar f\}$.  Since $f\in C^1$ and $A$ is compact, there is a
  neighbourhood $U$ of $A$ such that
  \[
    |f'|<\delta,
    \qquad
    f<\bar f+\delta
    \quad\text{on }U.
  \]
  As $f$ is not constant, we may choose $0<\varepsilon<\delta$ so that
  $\{f<\bar f+\varepsilon\}\Subset U$.

  Let $(a,b)$ be a connected component of this sublevel set meeting $A$,
  and choose $\theta_*\in(a,b)$ with $f(\theta_*)=\bar f$.  Since
  \[
    f(a)=f(b)=\bar f+\varepsilon>\bar f,
  \]
  the mean value theorem yields
  \[
    \theta_-\in(a,\theta_*),\qquad
    \theta_+\in(\theta_*,b)
  \]
  with $f'(\theta_-)<0<f'(\theta_+)$.  Since
  $[\theta_-,\theta_+]\subset U$, the remaining two estimates in
  \eqref{eq:choice_of_I} follow immediately.
\end{proof}

We are now in a position to prove the main result of this section. 
\begin{prop}
  \label{prop:confinement}
  There exist $\delta_0>0$ and $k_1\in\mathbb{N}$, depending only on $h$,
  $m_1$, $m_2$ and $\Omega$, such that if $I$ is chosen as in
  Lemma~\ref{lem:choice_of_I} with $\delta<\delta_0$, then for all
  $\theta_1,\theta_2\in I$ and all $k:\vert k\vert \ge k_1$ the arc $\sk$ is contained in
  $\overline\Omega$, meets $\partial\Omega$ only at its endpoints, and does so
  transversally.
\end{prop}

\begin{proof}
  Let $A=\{f=\bar f\}$ and $\bar{\xi}=\operatorname{arccosh}(\bar{f})$.  By Lemma~\ref{lem:cone_free}, there exist a closed
  neighbourhood $J$ of $A$ and $\alpha_1<\pi/2$ such that
  $|\alpha_h^\pm|\le\alpha_1$ on $J$.  Fix
  \[
    \alpha_0:=\frac{\alpha_1+\pi/2}{2}<\frac{\pi}{2},
  \]
  and let $t_1>0$ be given by Lemma~\ref{lem:no_return}.  We also set
  \begin{equation}
  	\label{eq:def_beta}
    \beta:=W_1(\bar\xi)-W_1(\bar\xi/2)>0.
  \end{equation}
  Since $A$ is compact, there exists $\delta_J>0$ such that
  \[
    \{f\le \bar f+\delta_J\}\subset J.
  \]
  
  Therefore, there exists an interval $I\subset J$ with the properties prescribed by Lemma \ref{lem:choice_of_I}.
  By Lemma~\ref{lem:rho_k}, the estimates of
  Lemma~\ref{lem:transit_time}, and \eqref{eq:rho_k_exp}, after increasing
  $k_1$ if necessary we may assume, uniformly in
  $\theta_1,\theta_2\in I$, that $K_0(\sk)<\bar K_0$ and that the unique
  turning point of $\sk$ satisfies
  \[
    \xi_{\min}\le\frac{\bar\xi}{2}.
  \]
  Since $W_1$ is increasing on $(0,+\infty)$, along either branch of the
  orbit while $\xi\ge\bar\xi$ one has
  \[
    |\dot\xi|^2
      =2\bigl(W_1(\xi)-W_1(\xi_{\min})\bigr)
      \ge 2\beta.
  \]
  Thus, the orbit crosses the region between its endpoint and the level
  $\{\xi=\bar\xi\}$ with a uniform speed bounded away from zero.

  If $\theta\in I$, we have
  $f(\theta)\le\bar f+\delta$.  Since $f=\cosh\xi$, this implies
  \[
    \xi(\theta)\le \bar\xi+\Delta(\delta),
    \qquad
    \Delta(\delta):=
      \operatorname{arcosh}(\bar f+\delta)-\bar\xi
      \longrightarrow0
  \]
  as $\delta\to0$.  Hence, the time needed to reach $\{\xi=\bar\xi\}$ is at
  most
  \[
    \frac{\Delta(\delta)}{\sqrt{2\beta}}.
  \]
  We may therefore choose $\delta_0<\delta_J$ so that this quantity is
  smaller than $t_1$ whenever $\delta<\delta_0$. 
  
  Proposition~\ref{prop:approximation_stable_manifold} shows that, uniformly
  in the endpoints, the initial velocity of $\sk$ converges as $k\to\infty$
  to the corresponding stable direction.  Enlarging $k_1$ once more, its
  angle with the inward normal is therefore at most $\alpha_0$.
  Lemma~\ref{lem:no_return} then implies that the initial part of the arc
  remains in $\Omega$ until it reaches $\{\xi=\bar\xi\}$.  From that moment
  until the second crossing of this level the arc lies in
  $ E_{<\bar f}\subset\Omega$.
  
  Applying the same argument to the time-reversed arc controls the final
  portion up to $\gamma(\theta_2)$.  Hence $\sk$ is contained in
  $\overline\Omega$ and meets $\partial\Omega$ only at its endpoints.
  The same uniform angle bound $\alpha_0<\pi/2$ also yields transversality
  at both impact points.
\end{proof}

\subsection{Generating functions}

We now recast the approximation result in terms of the first variation of the
Jacobi--Maupertuis length with respect to the endpoints.

\begin{definition}
  \label{def:generating}
  Fix $\theta_0\in\mathbb S^1$ and $k\in\mathbb Z$.  The \emph{$k$-th
  generating function} is
  \begin{equation}
    \label{eq:def_generating}
    S^k_{\theta_0}(\theta_1,\theta_2)
      :=\mathcal L_h\bigl(\sk\bigr),
    \qquad
    (\theta_1,\theta_2)\in
    Q_{\theta_0}:=
    \bigl(\mathbb S^1\setminus\{\theta_0\}\bigr)^2 .
  \end{equation}
\end{definition}

Thus $S^k_{\theta_0}(\theta_1,\theta_2)$ is the Jacobi--Maupertuis length of
the unique minimising arc joining $\gamma(\theta_1)$ to
$\gamma(\theta_2)$ in the homotopy class prescribed by the winding number
$k$.

By Proposition~\ref{prop:uniqueness}, the minimiser defining
$S^k_{\theta_0}$ is unique.  For $k\ne0$, the corresponding lifted
endpoints are distinct, so the distance is smooth in the endpoints;
hence $S^k_{\theta_0}$ is at least $\mathcal C^1$, with the same regularity
as $\gamma$.  The case $k=0$ is similar: one has only to avoid the diagonal. Howevver, this will not be
used below.

For every non-constant minimiser, the first variation formula yields
\begin{equation}
  \label{eq:first_variation}
  \partial_1S^k_{\theta_0}
   =-\bigl\langle \hat v_1,\dot\gamma(\theta_1)\bigr\rangle_{g_h},
  \qquad
  \partial_2S^k_{\theta_0}
   =\bigl\langle \hat v_2,\dot\gamma(\theta_2)\bigr\rangle_{g_h},
\end{equation}
where $\hat v_1$ and $\hat v_2$ denote the $g_h$-unit tangent vectors to
$\sk$ at its initial and final endpoints, respectively.

\begin{prop}
  \label{prop:generating_convergence}
  Let $h\ge0$ and $\theta_0\in\mathbb{S}^1$ be fixed. Then, as
  $\vert k\vert\to\infty$,
  \[
    \sup_{(\theta_1,\theta_2)\in Q_{\theta_0}}
    \Bigl\vert\,\partial_1S^k_{\theta_0}(\theta_1,\theta_2)
      +\bigl\langle \hat X^h_{\epsilon},\dot\gamma(\theta_1)\bigr\rangle_{g_h}\Bigr\vert
    +\Bigl\vert\,\partial_2S^k_{\theta_0}(\theta_1,\theta_2)
      +\bigl\langle \hat X^h_{-\epsilon},\dot\gamma(\theta_2)\bigr\rangle_{g_h}\Bigr\vert
    \ \longrightarrow\ 0 ,
  \]
  where $\epsilon=\mathrm{sgn}\,k$ and $\hat X^h_\pm$ denotes the
  $g_h$-normalisation of the field \eqref{eq:def_Xpm}. The convergence is
  exponential in $\vert k\vert$, with the rate of Lemma~\ref{lem:rho_k}.
\end{prop}

\begin{proof}
  By Proposition~\ref{prop:approximation_stable_manifold} the initial velocity of
  $\sk$ converges to that of $s^{\theta_1}_\infty$, that is to
  $X^h_{\epsilon}$, and the initial velocity of $\skinv$ converges to that of
  $s^{\theta_2}_{-\infty}$, that is to $X^h_{-\epsilon}$; hence the \emph{final}
  velocity of $\sk$ converges to $-X^h_{-\epsilon}$. Normalising and inserting
  into \eqref{eq:first_variation} gives the two limits.
\end{proof}

This formulation is tailored to the variational argument below.  If we try to concatenate multiple $\sk$ arcs to find billiard trajectories, the derivative of the
action with respect to an impact point $\theta$ is the sum of the contributions from the
incoming and outgoing arcs.  Proposition~\ref{prop:generating_convergence}
shows that, in the large-winding limit, this derivative converges to
\[
  -\bigl\langle \hat X^h_+ + \hat X^h_-,
    \dot\gamma(\theta)\bigr\rangle_{g_h}.
\]
Since $X^h_+$ and $X^h_-$ have the same pointwise norm, their $g_h$ normalisations involve the same positive factor.  Hence the limiting
stationarity condition is equivalent to
\[
  \bigl\langle \hat X^h_+ + \hat X^h_-,
    \dot{\gamma}(\theta)\bigr\rangle=0,
\]
which, by Lemma~\ref{lemma:critical_point}, is precisely the condition
$f'(\theta)=0$.

\section{Symbolic dynamics and non-integrability}
\label{sec:symbolic}

In this section we complete the variational construction of billiard orbits shadowing $W^s(\Gamma_h)$ and $W^s(\Gamma_h^{-1})$ and deduce the main dynamical consequences that this construction entails.

Proposition~\ref{prop:generating_convergence} shows that the
critical-point equations for chains of $\sk$ having $\vert k\vert$ large are a small
perturbation of the equation $f'=0$.  On the interval
$I=[\theta_-,\theta_+]$ chosen in Lemma~\ref{lem:choice_of_I}, the resulting
sign condition allows us to apply the Poincar\'e--Miranda theorem
\cite{Miranda1940,Kulpa1997}.  Its finite-dimensional form yields periodic
chains, while the countable-product version proved below gives bi-infinite
itineraries and hence the symbolic dynamics.  A one-sided variant of the same
construction will then be used to rule out non-constant real-analytic first
integrals.  Finite subalphabets provide the compact invariant sets needed for
the entropy estimates.

The standing assumptions are those of Section~\ref{sec:approximation}: $\Omega$
is a bounded domain with $\mathcal{C}^1$ boundary, containing $[c_1,c_2]$, whose
boundary is not a confocal ellipse; and $h\ge0$. 

\subsection{Standing choices}
\label{subsec:standing_choices}

Several constants enter the construction, and the order in which they are fixed
matters; we record it here once and for all in five items. Each item depends only on the previous ones.

\begin{enumerate}
  \item[(C1)] The arc $J$, the angle $\alpha_0<\pi/2$, the time $t_1$, the
  constant $\beta$ of \eqref{eq:def_beta} and the threshold $\delta_0$ are given
  by Proposition~\ref{prop:confinement}. They depend only on $h$, $m_1$, $m_2$
  and $\Omega$.
  \item[(C2)] A parameter $\delta<\delta_0$ is chosen, and with it the interval
  $I=[\theta_-,\theta_+]$ of Lemma~\ref{lem:choice_of_I}, satisfying
  \eqref{eq:choice_of_I}. A base point $\theta_0\notin I$ is fixed.
  \end{enumerate}
  Note that  $t_1$ and $\beta$ of (C1) are independent of $\delta$,
  which is what makes Proposition~\ref{prop:confinement} available for every
  sufficiently small $\delta$.
  
  We introduce now another parameter, essential for the application of the Poincar\'e--Miranda theorem. Recall that, by Lemma~\ref{lemma:critical_point}(ii) and
  \eqref{eq:sum_Xpm} the function
  \begin{equation}
  	\label{eq:def_Upsilon}
  	\Upsilon(\theta):=-\bigl\langle \hat X^h_++\hat X^h_-,\
  	\dot\gamma(\theta)\bigr\rangle_{g_h}
  \end{equation}
  is a positive multiple of $f'(\theta)$ at every $\theta$ (compare with \eqref{eq:function_ellipse} and \eqref{eq:sum_Xpm}); in particular
  $\Upsilon$ and $f'$ vanish simultaneously and have the same sign. Since
  $f'(\theta_-)<0<f'(\theta_+)$ by \eqref{eq:choice_of_I}, we may set
  \begin{equation}
  	\label{eq:eps}
  	2\ve:=\min\bigl\{\vert\Upsilon(\theta_-)\vert,\
  	\vert\Upsilon(\theta_+)\vert\bigr\}>0 ,
  \end{equation}
  and then choose, by Proposition~\ref{prop:generating_convergence}, an integer
  $k_0$ such that
  \begin{equation}
  	\label{eq:approximation_gradients}
  	\sup_{(\theta_1,\theta_2)\in I\times I}
  	\Bigl\vert\,\partial_1S^{k}_{\theta_0}(\theta_1,\theta_2)
  	+\bigl\langle\hat X^h_+,\dot\gamma(\theta_1)\bigr\rangle_{g_h}\Bigr\vert
  	+\Bigl\vert\,\partial_2S^{k}_{\theta_0}(\theta_1,\theta_2)
  	+\bigl\langle\hat X^h_-,\dot\gamma(\theta_2)\bigr\rangle_{g_h}\Bigr\vert
  	\ \le\ \frac{\ve}{2}
  \end{equation}
  for every $k\ge k_0$.
  
  We can now introduce the last two standing choices for the rest of the paper. 
  \begin{enumerate}
  \item[(C3)] The tolerance $\ve>0$ depends on the choice of $I$ and is given by \eqref{eq:eps}.
  \item[(C4)] The threshold $k_0$ is chosen so that
  \eqref{eq:approximation_gradients} holds and $k_0\ge k_1$, with $k_1$ as in
  Proposition~\ref{prop:confinement}. 
  \item[(C5)] The alphabet is defined as
  \begin{equation}
    \label{eq:def_alphabet}
    \mathcal{S}:=\{k\in\mathbb{N}:\ k\ge k_0\}.
  \end{equation}
\end{enumerate}

Note also that
$k_0$ depends on $I$, hence on $\Omega$ and $h$: this is the origin of the
$k_0=k_0(\Omega,h)$ appearing in Theorem~\ref{thm:A}.
Note, moreover, that by Proposition~\ref{prop:confinement} and (C4), for all $\theta_1,\theta_2\in I$
and all $k\in\mathcal{S}$ the arc $\sk$ is a genuine billiard arc: it lies in
$\overline\Omega$ and meets $\partial\Omega$ transversally at its two endpoints
only. This is used silently from now on.

\begin{remark}
  \label{rem:one_sided_alphabet}
  The alphabet \eqref{eq:def_alphabet} consists of windings of one sign only, and
  this is not a convention but a necessity. At an impact point $\theta_i$ the two
  limiting fields are $\hat X^h_{-\epsilon_{i-1}}$, coming from the incoming arc,
  and $\hat X^h_{\epsilon_i}$, coming from the outgoing one, where
  $\epsilon_j=\mathrm{sgn}\,\sigma_j$. They combine into
  $\hat X^h_++\hat X^h_-$, and hence into $\nabla f$ by
  \eqref{eq:def_Upsilon}, \emph{only if} $\epsilon_{i-1}=\epsilon_i$. For a
  sequence with mixed signs one gets $2\hat X^h_+$ or $2\hat X^h_-$ instead,
  which is not related to $f$. Replacing $\mathcal{S}$ by $\{k\le-k_0\}$ gives
  the mirror statement.
\end{remark}

\subsection{Chains and the Poincar\'e--Miranda Theorem}

Given a sequence $\sigma$ indexed by a set $J'\subseteq\mathbb{Z}$ of
consecutive integers, with $\sigma_i\in\mathcal{S}$, and a configuration
$\theta=(\theta_i)$ of points of $I$, we consider the chain whose $i$-th arc is
$s^{\theta_i,\theta_{i+1}}_{\sigma_i,\theta_0}$.  We define the  action as the formal sum:

\begin{equation}
  \label{eq:def_action}
  \mathcal{A}_\sigma(\theta)=\sum_{i}S^{\sigma_i}_{\theta_0}(\theta_i,\theta_{i+1}).
\end{equation}
By \eqref{eq:first_variation}, the stationarity equation at an interior node is
\begin{equation}
  \label{eq:def_G}
  G_i(\theta):=\partial_{\theta_i}\mathcal{A}_\sigma
   =\partial_2S^{\sigma_{i-1}}_{\theta_0}(\theta_{i-1},\theta_i)
    +\partial_1S^{\sigma_i}_{\theta_0}(\theta_i,\theta_{i+1})
   =\bigl\langle \hat v^{\,\mathrm{in}}_i-\hat v^{\,\mathrm{out}}_i,
                 \dot\gamma(\theta_i)\bigr\rangle_{g_h},
\end{equation}
where $\hat v^{\,\mathrm{in}}_i$ and $\hat v^{\,\mathrm{out}}_i$ are the
$g_h$-unit arrival and departure velocities at $\gamma(\theta_i)$.
At a fixed boundary point the metric $g_h$ is conformal to the Euclidean
metric.  Since the incoming and outgoing vectors are unit vectors and lie on
opposite sides of the tangent line, equality of their tangential components is
equivalent to specular reflection.  Hence $G_i(\theta)=0$ if and only if the
elastic reflection law holds at $\gamma(\theta_i)$. Thus:
\begin{center}
  \emph{a configuration $\theta\in I^{J'}$ with $G_i(\theta)=0$ for every $i$ is
  precisely a billiard trajectory realising the itinerary $\sigma$,}
\end{center}
the arcs being genuine billiard arcs by Proposition~\ref{prop:confinement}.

Two features of $G_i$ make the infinite-dimensional argument work. First, $G_i$
depends only on the three variables $\theta_{i-1},\theta_i,\theta_{i+1}$, hence
it is continuous on $I^{J'}$ for the \emph{product} topology.  Second, its sign on
the two faces $\{\theta_i=\theta_\mp\}$ is prescribed, uniformly in all the other
variables. 
We prove now the second property and then discuss Poincar\'e--Miranda theorem and the continuity of the map $G$ in the product topology. 

\begin{lemma}
  \label{lem:sign_on_faces}
  For every $\sigma$ with values in $\mathcal{S}$, every $i$ and every
  $\theta\in I^{J'}$,
  \[
    \theta_i=\theta_-\ \Longrightarrow\ G_i(\theta)\le-\ve,
    \qquad
    \theta_i=\theta_+\ \Longrightarrow\ G_i(\theta)\ge \ve .
  \]
\end{lemma}

\begin{proof}
  By Remark~\ref{rem:one_sided_alphabet} both $\sigma_{i-1}$ and $\sigma_i$ lie
  in $\mathcal{S}$, so \eqref{eq:approximation_gradients} applies to both
  summands of \eqref{eq:def_G} and gives
  $\vert G_i(\theta)-\Upsilon(\theta_i)\vert\le\ve$. The claim follows from
  \eqref{eq:eps}, since $\Upsilon(\theta_-)\le-2\ve$ and
  $\Upsilon(\theta_+)\ge2\ve$. The bound is uniform in $\theta_{i-1}$,
  $\theta_{i+1}$ and in the rest of the configuration, because
  \eqref{eq:approximation_gradients} is.
\end{proof}

Recall the classical statement of
\cite{Miranda1940}: if $\mathcal{G}=(\mathcal{G}_1,\dots,\mathcal{G}_n)$ is continuous on a box
$\prod_{i=1}^n[a_i,b_i]$, with $\mathcal{G}_i\le0$ on the face $\{x_i=a_i\}$ and
$\mathcal{G}_i\ge0$ on $\{x_i=b_i\}$, then $\mathcal{G}$ vanishes somewhere. The following version,
which is all we need, requires no functional-analytic machinery: it follows from
the finite-dimensional case and Tychonoff's Theorem.

\begin{lemma}[Poincar\'e--Miranda Theorem on a countable product]
  \label{lem:miranda_infinite}
  Let $J'$ be a countable set, let $[a,b]\subset\mathbb{R}$ and let
  $Q=[a,b]^{J'}$ carry the product topology. Let $\mathcal{G}_j\colon Q\to\mathbb{R}$,
  $j\in J'$, be continuous, and assume there is $\ve>0$ such that, for every
  $j\in J'$ and every $\theta\in Q$,
  \[
    \theta_j=a\ \Longrightarrow\ \mathcal{G}_j(\theta)\le-\ve,
    \qquad
    \theta_j=b\ \Longrightarrow\ \mathcal{G}_j(\theta)\ge\ve .
  \]
  Then there exists $\theta\in Q$ with $\mathcal{G}_j(\theta)=0$ for every $j\in J'$.
\end{lemma}

\begin{proof}
  Since $J'$ is countable we may write $J'=\bigcup_{n\ge1}F_n$ with
  $F_1\subseteq F_2\subseteq\cdots$ finite. Fix $\theta^\ast\in Q$ and, for each
  $n$, consider
  \[
    \mathcal{G}^{(n)}\colon [a,b]^{F_n}\to\mathbb{R}^{F_n},
    \qquad
    \mathcal{G}^{(n)}_j(\vartheta)=\mathcal{G}_j\bigl(\vartheta,\theta^\ast_{J'\setminus F_n}\bigr),
    \quad j\in F_n,
  \]
  obtained by freezing the coordinates outside $F_n$. It is continuous, and by
  hypothesis $\mathcal{G}^{(n)}_j\le-\ve<0$ on $\{\vartheta_j=a\}$ and $\mathcal{G}^{(n)}_j\ge\ve>0$
  on $\{\vartheta_j=b\}$. By the finite-dimensional Poincar\'e--Miranda theorem there is
  $\vartheta^{(n)}$ with $\mathcal{G}^{(n)}(\vartheta^{(n)})=0$. Let $\theta^{(n)}\in Q$
  agree with $\vartheta^{(n)}$ on $F_n$ and with $\theta^\ast$ elsewhere, so that
  \begin{equation}
    \label{eq:approx_zero}
    \mathcal{G}_j\bigl(\theta^{(n)}\bigr)=0\qquad\text{for every }j\in F_n .
  \end{equation}
  By Tychonoff's Theorem $Q$ is compact, and metrisable because $J'$ is
  countable; hence a subsequence $\theta^{(n_m)}$ converges to some $\theta$.
  Fix $j\in J'$; for $m$ large $j\in F_{n_m}$, so \eqref{eq:approx_zero} gives
  $\mathcal{G}_j(\theta^{(n_m)})=0$, and continuity of $\mathcal{G}_j$ for the product topology
  yields $\mathcal{G}_j(\theta)=0$.
\end{proof}

\begin{remark}
  The hypothesis that each $\mathcal{G}_j$ be continuous for the \emph{product} topology is
  essential, and is exactly what fails for a generic map on $[a,b]^{J'}$. In our
  application it holds for the most elementary of reasons: by \eqref{eq:def_G}
  the function $G_i$ depends on three coordinates only. This is also why no
  compactness needs to be established by hand: Tychonoff's Theorem provides it,
  and metrisability turns nets into sequences.
\end{remark}

\subsection{The symbolic dynamics}

\begin{thm}
  \label{thm:shadowing_sequences}
  With the choices (C1)--(C4), for every bi-infinite sequence
  $\sigma\in\mathcal{S}^{\mathbb{Z}}$ there exists a complete billiard
  trajectory at energy $h$, all of whose impact points lie in $I$ and whose
  $i$-th arc is $s^{\theta_i,\theta_{i+1}}_{\sigma_i,\theta_0}$; in particular it
  winds $\sigma_i$ times around $[c_1,c_2]$ between the $i$-th and the
  $(i+1)$-th impact.
\end{thm}

\begin{proof}
  Apply Lemma~\ref{lem:miranda_infinite} with $J'=\mathbb{Z}$,
  $[a,b]=[\theta_-,\theta_+]$ and $G_i$ as in \eqref{eq:def_G}: the $G_i$ are
  continuous for the product topology because each depends on three coordinates
  only, and the sign conditions hold with the $\ve$ of \eqref{eq:eps} by
  Lemma~\ref{lem:sign_on_faces}. We obtain $\theta\in I^{\mathbb{Z}}$ with
  $G_i(\theta)=0$ for every $i$, which by \eqref{eq:def_G} is the reflection law
  at every impact; and the arcs are billiard arcs by
  Proposition~\ref{prop:confinement}.
\end{proof}

\begin{corollary}[Periodic orbits]
  \label{cor:periodic}
  For every finite sequence $\sigma=(\sigma_1,\dots,\sigma_n)$ with values in
  $\mathcal{S}$ there is an $n$-periodic point of the billiard map (the
  actual minimal period may divide $n$), all of whose impacts lie in $I$, and whose
  consecutive arcs are $s^{\theta_i,\theta_{i+1}}_{\sigma_i,\theta_0}$.
\end{corollary}

\begin{proof}
  Read the indices cyclically in $\mathbb{Z}/n\mathbb{Z}$ and apply the
  finite-dimensional Poincar\'e--Miranda theorem to $(G_1,\dots,G_n)$ on $I^n$; the sign
  conditions of Lemma~\ref{lem:sign_on_faces} are unchanged.
\end{proof}

With $M_h$ the phase space of the billiard map of
Definition~\ref{def:theta_alpha}, set
\begin{equation}
  \label{eq:def_Lambda}
  \Lambda_h:=\Bigl\{z\in M_h:\ \text{for every }n\in\mathbb{Z},\
  \mathcal{B}^{n}(z)\text{ has impact point in }I\text{ and winding in }
  \mathcal{S}\Bigr\}.
\end{equation}

\begin{corollary}[Coding and entropy]
  \label{cor:semiconjugacy}
  The winding code defines a continuous surjection
  \[
    \Pi\colon\Lambda_h\longrightarrow\mathcal S^{\mathbb Z}
  \]
  satisfying
  \[
    \Pi\circ\mathcal B=\mathfrak s\circ\Pi .
  \]
 
  For every finite subset $A\subset\mathcal S$ with $|A|\ge2$, however,
  \begin{equation}
    \label{eq:def_Lambda_A}
    \Lambda_h(A)
      :=\{z\in\Lambda_h:\Pi(z)\in A^{\mathbb Z}\}
  \end{equation}
  is a non-empty compact $\mathcal B$-invariant set, and
  \[
    \Pi|_{\Lambda_h(A)}\colon
    \Lambda_h(A)\longrightarrow A^{\mathbb Z}
  \]
  is a factor map onto the full shift.  Consequently,
  \[
    h_{\mathrm{top}}
      \bigl(\mathcal B|_{\Lambda_h(A)}\bigr)
      \ge \log|A|.
  \]
  In particular, $\mathcal B$ admits compact invariant subsystems with
  arbitrarily large topological entropy.
\end{corollary}

\begin{proof}
  Surjectivity follows from Theorem~\ref{thm:shadowing_sequences}, while
  equivariance is immediate from the definition of the coding.  Continuity is
  equally direct: each coordinate of $\Pi$ is the winding number of a
  collisionless arc and is locally constant under small perturbations with
  transversal endpoints, since the corresponding homotopy class in
  $\mathbb R^2\setminus[c_1,c_2]$ is preserved.  Hence $\Pi$ is continuous for
  the product topology.

  Fix a finite $A\subset\mathcal S$ and set
  \[
    \mathcal P_A:=\bigcup_{k\in A}\mathcal P_k,
    \qquad
    \mathcal P_k
      :=\bigl\{
          z_k(\theta_1,\theta_2):
          (\theta_1,\theta_2)\in I^2
        \bigr\},
  \]
  where $z_k(\theta_1,\theta_2)$ denotes the inward state whose first free arc
  is $s^{\theta_1,\theta_2}_{k,\theta_0}$.  The dependence on the endpoints is
  continuous, and Proposition~\ref{prop:confinement} gives transversal return
  uniformly on $I^2$; thus $\mathcal P_A$ is compact and the billiard map is
  continuous on a neighbourhood of it.

  Since every orbit in $\Lambda_h(A)$ remains in $\mathcal P_A$, compactness
  follows by a standard diagonal argument.  Indeed, from any sequence
  $z_n\in\Lambda_h(A)$ one may extract a subsequence such that
  $\mathcal B^j(z_n)$ converges in $\mathcal P_A$ for every $j\in\mathbb Z$.
  Continuity of the return map passes the orbit relations to the limit, yielding
  a complete orbit with all winding symbols in $A$.  Thus $\Lambda_h(A)$ is
  closed in the compact set $\mathcal P_A$, hence compact.

  The restricted coding is onto by
  Theorem~\ref{thm:shadowing_sequences}.  Since topological entropy does not
  increase under factor maps between compact dynamical systems,
  \[
    h_{\mathrm{top}}
      \bigl(\mathcal B|_{\Lambda_h(A)}\bigr)
      \ge
    h_{\mathrm{top}}(\mathfrak s|_{A^{\mathbb Z}})
      =\log|A|,
  \]
  which proves the claim.
\end{proof}

Since the sign conditions of Lemma~\ref{lem:sign_on_faces} are uniform in all
variables, the first impact point may be fixed without affecting the argument.
Thus, for any $\bar\theta\in I$, the remaining equations can be solved with
$\theta_1=\bar\theta$, producing orbits with prescribed itinerary issued from
the fibre
\begin{equation}
  \label{eq:def_fibre}
  F_{\bar\theta}
  :=\bigl\{z\in M_h:\ \pi(z)=\gamma(\bar\theta)\bigr\}.
\end{equation}
Since $\gamma(\bar\theta)$ is away from the centres, $F_{\bar\theta}$ is a half of the
non-degenerate velocity circle on the energy shell, hence a real-analytic
submanifold.

\begin{corollary}[One-sided itineraries with prescribed initial impact]
  \label{cor:one_sided}
   Fix $\bar\theta \in I$. Then:
  \begin{enumerate}
    \item[(i)] for every sequence
    $\sigma=(\sigma_1,\sigma_2,\ldots)\in\mathcal S^{\mathbb N}$, there exists a
    forward-infinite billiard trajectory at energy $h$, starting at
    $\gamma(\bar\theta)$, whose impact points lie in $I$ and whose $i$-th arc is
    \[
      s^{\theta_i,\theta_{i+1}}_{\sigma_i,\theta_0},
      \qquad \theta_1=\bar\theta;
    \]
    \item[(ii)] choosing one such trajectory for each $\sigma$ and denoting its
    initial state by $z(\sigma)\in F_{\bar\theta}$, the map
    \[
      \sigma\longmapsto z(\sigma)
    \]
    is injective. In particular, the set of such initial states is uncountable;
    \item[(iii)] as $\sigma_1\to+\infty$,
    \[
      z(\sigma)\longrightarrow
      \bigl(\tilde \gamma(\bar\theta),
      \,X^h_+(\tilde \gamma(\bar\theta))\bigr),
    \]
    uniformly with respect to $\sigma_2,\sigma_3,\ldots$;
  \end{enumerate}
\end{corollary}

\begin{proof}
  For (i), we apply Lemma~\ref{lem:miranda_infinite} to the variables
  $(\theta_i)_{i\ge2}$, with $\theta_1=\bar\theta$ fixed.  The sign conditions
  of Lemma~\ref{lem:sign_on_faces} are uniform in the remaining variables, so
  the same argument yields a solution of $G_i=0$ for every $i\ge2$.
  These are precisely the reflection conditions at all subsequent impacts.

  For (ii), the initial state uniquely determines the forward billiard
  trajectory and hence its sequence of winding numbers.  Distinct itineraries
  therefore give distinct initial states.  Since $\mathcal S$ contains at
  least two symbols, $\mathcal S^{\mathbb N}$ is uncountable.

  Statement (iii) follows directly from
  Proposition~\ref{prop:approximation_stable_manifold}, applied to the first arc
  $s^{\bar\theta,\theta_2}_{\sigma_1,\theta_0}$; the convergence is uniform in
  $\theta_2\in I$, and hence is independent of the  itinerary.

\end{proof}

\subsection{Analytic non-integrability}
\label{subsec:nonintegrability}

We conclude by showing how the orbits built in Corollary \ref{cor:one_sided}  give an obstruction to the existence of an analytic first integral in the neighborhood of the intersection between stable graphs and their reflection $\mathcal G_h^\pm\cap\iota (\mathcal G_h^{\pm})$.
In particular, no global analytic first
integral exists.

Let us denote by $M_h$ the set of inward-pointing vectors over $\partial \Omega$ and  $\mathcal{D} \subseteq \mathbb{S}^1\times (-\frac{\pi}{2},\frac{\pi}{2})\approx M_h$ the domain of the billiard map $\mathcal{B}$, i.e. the pairs $(\theta,\alpha)$ corresponding to inward pointing vectors over $\partial \Omega$ for which the first return time is finite and the final velocity is transverse to $\partial \Omega$. We give the following definition of local analytic integrability
\begin{definition}
	\label{def:first_integral}
	Let $\mathcal U\subset M_h$ be an open set.  A function
	$\mathcal F\colon\mathcal U\to\mathbb R$ will be called a
	\emph{local first integral of the billiard} if $\mathcal{F}(\mathcal{B}(u))=\mathcal{F}(u)$ for all $u \in \mathcal{D}\cap \mathcal U$ with $\mathcal B(u)\in\mathcal U$,  and $\mathcal{F}$ is real-analytic, as a function of the coordinates $\theta$ and $\alpha$ introduced in Definition \ref{def:theta_alpha}.
\end{definition}

If the set $\partial \Omega$ is analytic, then $M_h$ itself is  analytic and embedded. This implies that $\mathcal{B}$ is analytic as well. Moreover, any analytic function defined on $M_h$ extends to an analytic function in a small neighbourhood. Thus, the notion coincides with the standard analyticity definition in $\mathbb{R}^4$.

\begin{thm}[Local analytic non-integrability]
	\label{thm:local_nonintegrability}
	Let $\mathcal U\subset M_h$ be a connected neighbourhood of
	the graph $
	\mathcal{G}_h^+
	$ (resp. $\mathcal{G}_h^-$) given in Definition \ref{def:alpha_pm}. 
	Under the standing assumptions, every real-analytic local first integral
	$\mathcal F\colon\mathcal U\to\mathbb R$ is constant.
\end{thm}

\begin{proof}
     Let us observe first that $\mathcal{F}$ is constant on $\mathcal{G}_h^+$ in a neighbourhood of the minimum set of $f$. Indeed, Proposition \ref{prop:approximation_stable_manifold} implies that whenever $\sk\subseteq\bar{\Omega}$ for $k$ large enough, we can pass to the limit with respect to $k$ and obtain 
	\[
	\mathcal{F}(\theta_1,\alpha_h^+(\theta_1)) = \mathcal{F}(\theta_2,-\alpha_h^-(\theta_2)
	\]
	Moreover, if $\theta'$ is a critical point of $f$, it holds that $-\alpha_h^-(\theta') =\alpha_h^+(\theta'))$. Combining this with Proposition \ref{prop:confinement}, we conclude that, in a neighbourhood of a minimum point, the function $\mathcal{F}$ is constant  and equal to $c_+$ on $\mathcal{G}_h^+$. An analogous argument shows that $\mathcal{F}$ must be constant and equal to $c_-$ on $\mathcal{G}_h^-$.
	
	Pick $\bar{\theta} \in I$ so that the fiber over $\gamma(\bar\theta)$ has non-empty intersection with $\mathcal{U}\cap \mathcal{D}$.  The approximation of $\mathcal{W}^s(\Gamma_h)$ is uniform and so, after increasing the threshold on $k$ if necessary, the one-sided trajectories given in Corollary \ref{cor:one_sided} are entirely contained
	in $\mathcal U$.
	
	Let us consider a sequence 
	$\sigma\in\mathcal S^{\mathbb N}$ satisfying $\sigma_i\to+\infty$, i.e. such that the number of windings grows with the iterations. 
	For any such sequence there exists an initial condition $z(\sigma)\in F_{\bar\theta}$ given by
	Corollary~\ref{cor:one_sided}, giving a billiard orbit that realizes said sequence. 
	Since the arcs $\sk$ converge to the stable manifold, the value of the first integral must be equal to $c_+$. Since the fiber over $\bar \theta$ is analytic, and the restriction of $\mathcal{F}$ to it attains the value $c_+$ infinitely many times and the zeros accumulate at the intial condition corresponding to the stable graph. Thus, $\mathcal{F}$ it is constant.  Since $\bar \theta$ is arbitrary, we conclude that $\mathcal{F}$ must be constant in a neighbourhood of the minimum set and thus everywhere on $\mathcal{U}$. An analogous argument settles the case of $\mathcal{G}_h^-$.
\end{proof}

Finally, let us notice that the
non-integrability theorem is logically independent of the entropy statement:
the former uses one-sided itineraries with divergent winding numbers, whereas
the latter arises from bi-infinite symbolic sequences.

	\bibliographystyle{plain}
	\bibliography{rigidity_2center}

	\bigskip
	
	\noindent
	S. Baranzini\\
	Universit\`a  San Raffaele di Roma\\
	Via di Val Cannuta 247, 00166 Roma, Italy\\
	\vspace{-0.2cm}
	
	\noindent
	Ruhr-Universit\"at Bochum\\
	Universit\"atsstra\ss e 150, 44801 Bochum S\"ud, Germany \\
	\texttt{stefano.baranzini@uniroma5.it}

	\vspace{0.8cm}
	
	\noindent
	S. Terracini \\
	Dipartimento di Matematica ``Giuseppe Peano'', Universit\`a degli Studi di Torino\\
	Via Carlo Alberto 10, 10123 Torino, Italy\\
	\texttt{susanna.terracini@unito.it}

\end{document}